\documentclass[11pt]{article}
\usepackage{graphicx}
\usepackage{amsmath}
\usepackage{amssymb}
\usepackage{amsfonts}
\usepackage{amsthm}
\usepackage{xcolor}
\usepackage{caption}
\usepackage{enumerate}
\usepackage{ulem}
\usepackage[left=2.1cm,top=1.5cm,right=2.1cm, bottom=2.2cm,letterpaper]{geometry}
\def\HH{\mathcal{H}}
\def\R{\mathbb{R}}
\def\X{\mathbb{X}}

\def\PP{\mathcal{P}}
\def\CP{\mathcal{F}}

\def\e{{\bf e}}
\def\U{{\bf u}}
\def\f{{\bf f}}
\def\0{{\mathbf 0}}
\def\Q{{\bf v}}

\def\phi{\varphi}

\newtheorem{theorem}{Theorem}
\newtheorem{corollary}{Corollary}
\newtheorem{lemma}{Lemma}
\newtheorem{proposition}{Proposition}

\newtheorem{remark}{Remark}
\numberwithin{equation}{section}
\usepackage{hyperref}
\hypersetup{
colorlinks=true,
linkcolor=blue,
filecolor=magenta,
urlcolor=cyan,
}
\begin{document}
\title{On the Uniqueness and Structural Stability of Couette-Poiseuille Flow\\ in a Channel for Arbitrary Values of the Flux}
\author{G.P. Galdi - F. Gazzola - M. Korobkov - X. Ren - G. Sperone}

\date{}
\maketitle

\begin{abstract}
We establish uniqueness and structural stability of a class of parallel flows in a 2D straight, infinite channel, under perturbations with  either globally or locally bounded Dirichlet integrals. The significant feature of our result is that it does not require any restriction on the size of the flux characterizing the flow. Precisely, by extending and refining an approach initially introduced by J.B.~McLeod, we demonstrate the continuous invertibility of the linearized operator at a generic Couette-Poiseuille solution that does not exhibit flow reversal.  We then deduce  local uniqueness of these solutions as well as their  nonlinear structural stability  under small external forces. Moreover, we prove the uniqueness of certain class of Couette-Poiseuille solutions ``in the large," within the set of solutions possessing natural symmetry. Finally, we bring an example showing that, in general, if the flow reversal assumption is violated,  the linearized operator is no longer invertible.\par\medskip
\end{abstract}

\section{Introduction}
A ``distorted channel," $\mathcal C$, is a two-dimensional domain consisting of two semi-infinite straight channels, $\mathcal C_1$ and $\mathcal C_2$, connected smoothly by a bounded and regular domain, $\mathcal C_0$. One of the most intriguing and still unresolved problems in the mathematical theory of the Navier-Stokes equations is to prove (or disprove) for such domains the existence of a steady-state solution, that converges at large distances from $\mathcal C_0$ to arbitrarily assigned  Poiseuille flows in $\mathcal C_1$ and $\mathcal C_2$ \cite[Chapter XIII]{G}. Since these flows are characterized by their (equal) flux $\Phi$ through the generic cross-section of $\mathcal C$, the problem can be reformulated by asking whether  solutions exist for an {\em arbitrary} non-zero value of $\Phi$ that satisfy the aforementioned asymptotic conditions.\footnote{An analogous problem can be formulated in the three-dimensional case, with channels replaced by pipes \cite[\S\, XIII.3]{G}.} This question, originally posed by J. Leray to O.A. Ladyzhenskaya in the late 1950s (see \cite[Remark 1.6]{G1}), is currently known as the ``Leray problem."
\par
The first contribution to Leray problem is due to Amick \cite{AM1,AM2}. He looks for solutions whose velocity field $\mathbf{u}$ is of the form
$$
\mathbf{u}=\mathbf{u}_*+\mathbf{v}
$$
where $\mathbf{u}_*$ is a smooth (given) extension of the Poiseuille flows to the whole $\mathcal C$, and $\mathbf{v}$ is a ``correction" satisfying a suitable perturbed nonlinear problem around $\mathbf{u}_*$. This construction produces existence, with $\mathbf{v}$ in a subspace of the Sobolev space $H^1(\mathcal C)$, on condition that $|\Phi|$ is not too large \cite{AM1}.
A few years later, Ladyzhenskaya \& Solonnikov \cite{LS1} addressed, among others, Leray problem from a different perspective. Namely, they looked for solutions with $\mathbf{u}$ in a subspace of the {\em local} Sobolev space $H^1(\mathcal C_{a,a+1})$, $a\in\mathbb R$, where $\mathcal C_{a,a+1}$ is a cross-sectional slice of $\mathcal C$ located at $a$ and of thickness 1; see \eqref{ex}. In this way, they were able to prove existence of solutions with $\U\in H^1(\mathcal C_{a,a+1})$,  for arbitrary $|\Phi|$. However, the asymptotic convergence of such $\U$ to the associated Poiseuille velocity fields in $\mathcal C_{i}$, $i=1,2$, is only guaranteed if $|\Phi|$ is not too large.\footnote{For other contributions to Leray problem, we refer to \cite[\S\, XIII.3, and Notes to Chapter XIII]{G} and \cite{pileckas}.}
\par
The fact that both approaches presented in \cite{AM1} and \cite{LS1} require, for convergence to the Poiseuille flow, that the flux magnitude not be too large, raises the question of whether, for sufficiently large flux, there may exist corresponding bounded solutions different from the Poiseuille solutions, and whether it is precisely to the manifold of velocity fields of these other solutions that $\U$ might converge.

\par
The question of the local uniqueness of Poiseuille flow --namely, absence of other solutions in a neighborhood of this flow--  has been investigated by several authors in \cite{rabier1,rabier2,SWX1,SWX2}. Postponing a detailed description of their results to a later point, we will simply state here that, in the general case, they all give a positive answer but on condition  that the magnitude of the flux remains below a certain constant or else is above a suitable constant.
\par
The main objective of this paper is to prove a rather comprehensive result on the local uniqueness of a class of parallel flows in a two-dimensional straight, infinite channel. Our results guarantee, in particular, that for {\em arbitrary} values of the flux $\Phi$, the Poiseuille solution is, locally, the only possible one in both the functional settings of Amick \cite{AM1} and Ladyzhenskaya \& Solonnikov \cite{LS1}. Moreover, such uniqueness property is proved to hold globally, in the subclass of solutions possessing suitable symmetries.
\par
To present our results and the main ideas underlying our work, we begin to give the precise formulation of the problem. Let $S=\R\times(-1,1)$ denote the infinite channel, and consider the following boundary-value problem:
\begin{equation}\label{poiseuille2d}
\left\{\begin{array}{ll}
-\Delta\U+ (\U\cdot\nabla)\U+\nabla p=\0\quad\mbox{in }S,\\[6pt]
\nabla\cdot\U=0\quad\mbox{in }S,\\[6pt]
\U=(3A-B+C) {\e_1}\mbox{ for }y=-1\, ,\quad \U=(3A+B+C) {\e_1}\mbox{ for }y=1,\\[6pt]
\displaystyle\lim_{|x|\to \infty}\U(x,y)=\CP(y)\e_1 \quad {\text{uniformly in}} \ y \in [-1,1]\, ,
\end{array}\right.
\end{equation}
where $\U=(u_1,u_2)$ is the fluid velocity vector field, $p$ its scalar pressure and
\begin{equation}\label{parabolasign2}
\CP(y)\doteq 3Ay^2+By+C \quad\forall y\in[-1,1] \, , \quad \text{for some} \ (A,B,C) \in \mathbb{R}^3 \setminus \{(0,0,0)\} .
\end{equation}
In \eqref{poiseuille2d} we have set the viscosity coefficient equal to $1$, since its actual value is irrelevant in our analysis.
Notice that, for any choice of $A,B,C\in \mathbb R$, problem \eqref{poiseuille2d} admits the  parallel-flow solution
\begin{equation}\label{Poiseuille}
\U_{*}(x,y)=\CP(y)\e_1\, ,\quad p_{*}(x,y)=6A x\, ,\qquad(x,y)\in S\,.
\end{equation}
Without loss of generality, we can assume $A\le 0$. Special choices of $A,B,C$  give different flows, such as
\begin{equation}\label{PF}
\begin{aligned}
& A=-\mbox{$\frac13$}C<0\,,\ B=0  \ \Longrightarrow \ \CP(y)=-3A(1-y^2) \ (\text{Poiseuille flow}) , \\[5pt]
& A=0\,,\ B=C \ \Longrightarrow \ \CP(y)=B\,(1+y) \ (\text{Couette flow}), \\[5pt]
& A=B=0 \ \Longrightarrow \ \CP(y)\equiv C \ (\text{constant flow}).
\end{aligned}
\end{equation}
The question we address is whether \eqref{Poiseuille} is, in a suitable class, the only solution to  \eqref{poiseuille2d}.
To this end, set
\begin{equation}\label{Fl}
	\Phi \doteq 2(A +C).
\end{equation}
and  replace \eqref{poiseuille2d}$_4$ with the (much weaker) flux condition\footnote{Notice that by \eqref{poiseuille2d}$_{2,3,4}$,
condition \eqref{fl} is equivalent to $\int_{-1}^1 \U(x,y) \cdot \e_1 dy = \Phi$, for all $x\in \R$.}
\begin{equation}\label{fl}
	\int_{-1}^1 \U(0,y) \cdot \e_1 dy = \Phi\,.
\end{equation}
We then look for solutions  to \eqref{poiseuille2d}$_{1,2}$ in the form
\begin{equation}\label{form}
\U(x,y)=\U_{*}(x,y)+\Q(x,y) \quad \text{and} \quad p(x,y)=p_{*}(x,y) + q(x,y) \quad\forall(x,y)\in S \, ,
\end{equation}
where the pair $(\Q, q)$ satisfies the following problem:
\begin{equation}\label{poiseuille2dqq}
	\left\{\begin{array}{ll}\smallskip
		-\Delta\Q+ (\Q\cdot\nabla)\Q +  (\Q\cdot\nabla)\U_{*} +  (\U_{*}\cdot\nabla)\Q +\nabla q=\0\quad\mbox{in }S,\\[6pt] \smallskip\nabla\cdot\Q=0\quad\mbox{in }S,\\[6pt]
		\Q=\0 \ \mbox{  for  }\ y= \pm 1, \\[6pt]
		\displaystyle\int_{-1}^1 \Q(0,y) \cdot \e_1 dy = 0.
	\end{array}\right.
\end{equation}
We emphasize that there are {\em no convergence conditions} like \eqref{poiseuille2d}$_3$ at the (infinite)
inlet/outlet.
\par
The main objective is therefore to demonstrate that problem \eqref{poiseuille2dqq} admits only the trivial solution $\mathbf{v}\equiv\nabla q\equiv \mathbf{0}$, at least for ``small"  $\Q$, within a class of functions similar to those considered by Amick and Ladyzhenskaya \& Solonnikov.
A natural way of showing this is to prove that the linearized problem:
\begin{equation}\label{poiseuille2dqqlinear}
	\left\{\begin{array}{ll}\smallskip
		-\Delta\Q+ (\Q\cdot\nabla)\U_{*} +  (\U_{*}\cdot\nabla)\Q +\nabla q=\0 \quad\mbox{in }S,\\[6pt]
		\nabla\cdot\Q=0 \quad\mbox{in }S,\\[6pt]
		\Q=\0 \ \mbox{  for  } \ y= \pm 1, \\[6pt]
\displaystyle		\int_{-1}^1 \Q(0,y) \cdot \e_1 dy = 0,
	\end{array}\right.
\end{equation}
has only the trivial solution. This is ensured if the linear operator
$$
\mathcal{L} (\Q,q) \doteq -\Delta\Q+ (\Q\cdot\nabla)\U_{*} +  (\U_{*}\cdot\nabla)\Q +\nabla q \, ,
$$
referred to as the \textit{Couette-Poiseuille linearization},
is, suitably defined, an isomorphism.
 In which case, a contraction-mapping argument, for example, will guarantee the desired uniqueness property in the appropriate classes.
\par
The first rigorous contribution to the study of this problem is due to Rabier \cite{rabier1,rabier2}, in the case where $\U_*$ is the Poiseuille flow \eqref{Poiseuille}--\eqref{PF}$_1$.\footnote{As is known, there is extensive literature that addresses this topic using appropriate approximations. In this regard, we refer the reader to the classic book by Drazin \& Reid \cite{DR} and the bibliography contained therein.}  Specifically, in \cite{rabier1} the continuous invertibility of the  linearization $\mathcal L$ for arbitrarily large flux is proved in the Sobolev spaces of functions enjoying appropriate symmetry properties. This result provides, as a byproduct, the local existence and uniqueness, for an arbitrary flux, of a symmetric solution to the problem obtained from \eqref{poiseuille2dqq} by replacing the zero terms on the right-hand side of equations \eqref{poiseuille2dqq}$_1$ and \eqref{poiseuille2dqq}$_3$ with prescribed "small" fields (external force and boundary velocity); see \cite[Corollary 17]{rabier1}.
The symmetry assumption was dropped in subsequent paper \cite{rabier2}, but only if the magnitude of the flux is suitably restricted. More recently, Sha, Wang \& Xie have provided a further contribution to the problem in the subclass of solutions that are periodic in the $x$-direction and, again as in \cite{rabier1,rabier2}, when $\U_*$ is the Poiseuille flow  \eqref{Poiseuille}--\eqref{PF}$_1$. Their results furnish the continuous invertibility of $\mathcal L$ for arbitrarily large
flux but sufficiently small period \cite{SWX1}, or with \textit{suitably} large flux but arbitrary period \cite{SWX2}.
\par
In this paper, we prove the  isomorphism property of the Couette-Poiseuille linearization around a given parallel flow \eqref{Poiseuille}, without imposing any restrictions on the magnitude of the flux or symmetry requirements. More precisely, under suitable assumptions on $A,B,C$ that prevent flow reversal (see \eqref{ABC} below), we show the following properties.
\begin{itemize}
\item[(1)]  The operator $\mathcal{L}$  is a continuous bijection from
\begin{equation}
\mathcal H(S)\dot=\{ \Q \in H^2(S) \cap H^1_0(S) \ | \ \nabla \cdot \Q = 0 \ \text{in} \ S \}  \times \left( \{q \in H^1_{\text{loc}}(S) \ | \  \nabla q \in L^2(S) \}/\R\right)\label{calh}
\end{equation}
onto $L^2(S)$; see Theorem \ref{L2}.  Note that $\Q \in H^1_0(S)$ and $\nabla \cdot \Q = 0$ together imply the zero-flux condition \eqref{poiseuille2dqqlinear}$_4$.
\item[(2)] We introduce the ``local" norms (with $m =0, 1,2,\cdots$)
\begin{equation}
	\|f\|_{\X^m} \doteq \sup_{a \in \mathbb{R}} \|f\|_{H^m((a,a+1) \times (-1,1))},\label{ex}
\end{equation}
which define the corresponding Banach spaces $\X^m \subset H^m_{\text{loc}}(S)$.\footnote{We use the same notation for spaces of scalar functions and vector fields.} The operator $\mathcal{L}$ is then a continuous bijection from
\begin{equation}\label{bbh}
\mathbb H(S)\dot=\{ \Q \in \X^2 \ | \ \Q \ \mbox{satisfies} \ \eqref{poiseuille2dqqlinear}_{2,3,4}\}  \times \left( \{q \in H^1_{\text{loc}}(S) \ | \  \nabla q \in \X^0 \,\} / \R \right)
\end{equation}
onto $\X^0$; see Theorem \ref{bounded}. Note that the Couette--Poiseuille flow \eqref{Poiseuille} belongs to the class $\mathbb H(S)$.
\end{itemize}
\par
Combining these results with the  Contraction-Mapping Theorem, we show that the nonlinear problem \eqref{poiseuille2dqq} does not admit nontrivial solutions in some $H^2(S)$-neighborhood (in case (1)) and $\X^2$-neighborhood (in case (2))  of the origin $(\Q, q) = (\0, 0)$; see Corollaries \ref{Cor1} and \ref{Cor2}.  Moreover, in Theorem  \ref{nosymm} we also prove that  \eqref{poiseuille2dqq} does not admit nontrivial solutions in that subspace, $\mathbb S$ (say), of $\mathbb X^2$ of functions possessing a suitable symmetry. Since this result holds in an {\em arbitrary} neighborhood of $(\Q, q) = (\0,0)$ and since the Poiseuille flow is in $\mathbb S$ and satisfies the assumption \eqref{ABC}, we obtain, in particular, the following important property:  {\em Given an arbitrary value of the flux, the corresponding Poiseuille flow is the only solution to \eqref{poiseuille2d}   in the class $\mathbb S$}; see Corollary \ref{Pos}.
\smallskip\par
The next, and final, question we address concerns the relevance of the assumptions  \eqref{ABC} in the proof of the above results. In Theorem \ref{prop:general} we show, by means of an example, that they are indeed necessary. Precisely:
\begin{itemize}
\item[(3)] If the assumptions  \eqref{ABC} on $A,B,C$ are violated, that is, flow reversal is allowed, the operator $\mathcal{L}$ 
can fail to be continuously invertible in both the classes mentioned in (1) and (2).
\end{itemize}
\par
The proof of the properties  reported in (1) and (2) above relies on the careful study of the non-homogeneous Orr-Sommerfeld equation \eqref{OS}. This equation is   obtained by  applying the partial Fourier transform in the $x$-direction to the stream-function formulation of \eqref{poiseuille2dqqlinear}; see Section \ref{sec:2}. Our study, performed in Section \ref{nobranch1} (for case (1)) and in Section \ref{sec:5} (for case (2)) is aimed at showing the isomorphism property of $\mathcal L$.
Our approach follows and generalizes a brilliant idea due to the late J.B.~McLeod which we discovered ``by chance" (see Section~\ref{hfacts} for a historical reconstruction). This idea  is contained in one of the chapters of a book dedicated to ordinary differential equations  \cite[Chapter 10]{HMC}, and this is why, we believe, it went unnoticed in the Mathematical Fluid Mechanics community.
\par
The outline of the paper is the following. In Section \ref{sec:2} we reformulate the linear problem \eqref{poiseuille2dqqlinear} in terms of the stream function. In Section \ref{sec:3},  we study the non-homogeneous Orr-Sommerfeld equation and show suitable  properties of existence and continuous dependence on the data; see Proposition \ref{prop:OS} and Lemma \ref{lem:poincare}. These key results are then employed in Section \ref{nobranch1} to show, in
Theorem~\ref{L2}, the isomorphism property of $\mathcal L$ when defined in the space $\mathcal H(S)$ and under the assumptions \eqref{ABC}. Combining Theorem \ref{L2} with a
contraction-mapping argument we then show, in Theorem \ref{cor1}, existence and uniqueness of solutions to the nonlinear problem obtained from \eqref{poiseuille2dqq} by perturbing the right-hand side of \eqref{poiseuille2dqq}$_1$ with a ``small" force in $L^2(S)$. As a corollary to this result, it follows that the Couette-Poiseuille flow is unique in the class \eqref{form} where $\Q\in\mathcal H(S)$ with a ``small" $H^1$-norm, for arbitrary values of the flux; see Corollary \ref{Cor1}. The first part of  Section \ref{sec:5}, is dedicated to the proof of the isomorphism property stated in (2). Under the assumptions \eqref{ABC}, this is accomplished in Theorem \ref{bounded}, by means of a detailed investigation of the local properties of solutions to \eqref{OS}.  Then, employing Theorem \ref{bounded} in conjunction with the Contraction-Mapping Theorem, we proceed as in Section \ref{nobranch1} to show, in Theorem \ref{cor2}, existence and uniqueness of solutions to the nonlinear problem obtained from \eqref{poiseuille2dqq} by perturbing the right-hand side of \eqref{poiseuille2dqq}$_1$ with a ``small" force in $\X^0$. As a consequence (see Corollary \ref{Cor2}), the Couette-Poiseuille flow is unique in the class \eqref{form} where $\Q\in\mathbb H(S)$ with a ``small" $\X^1$-norm, for arbitrary values of the flux. In the final Section \ref{sec:6} we investigate the results obtained in the previous one in the subclass $\mathbb S\subset \mathbb H(S)$ of solutions possessing ``natural" symmetry properties; see \eqref{Xy}. We thus show in Theorem \ref{nosymm}
that the hypothesis of smallness needed in Corollary \ref{Cor2} can be removed in the subclass $\mathbb S$, while still keeping assumption \eqref{ABC}. However, by combining the results of Theorem \ref{bounded} with those of \cite{GGN} in Theorem \ref{prop:general}, we present an explicit example that demonstrates that, if condition \eqref{ABC} is violated, the null space of $\mathcal L$ in $\mathbb{H}(S)$ is not trivial, which means that $\mathcal L$ ceases to be an isomorphism.

\section{Reformulation of the problem in terms of the stream function} \label{sec:2}
The components $\big(v(x,y),w(x,y)\big)$ of the perturbation $\Q$ in \eqref{form}, together with the new scalar pressure $q$, satisfy the system (subscript denotes differentiation with respect to the indicated variable)
\begin{equation}\label{nonlinear}
\left\{\begin{array}{ll}
\Delta v- (\CP v_x+\CP'w+vv_x+wv_y)=q_x & \mbox{in }S,\\[3pt]
\Delta w- (\CP w_x+vw_x+ww_y)=q_y & \mbox{in }S,\\[3pt]
v_x+w_y=0 & \mbox{in }S,\\[3pt]
v=w=0 & \mbox{on }\partial S, \\[3pt]
\int_{-1}^1 v(0,y) dy = 0.
\end{array}\right.
\end{equation}
By differentiating \eqref{nonlinear}$_1$ with respect to $y$ and \eqref{nonlinear}$_2$ with respect to $x$, and by using
\eqref{nonlinear}$_3$, we obtain the system
\begin{equation}\label{fullnonlin}
\left\{\begin{array}{ll}
\Delta v_y- (\CP v_{xy}+\CP''w+v_xv_y+vv_{xy}+v_yw_y+v_{yy}w)=q_{xy} & \mbox{in }S,\\[3pt]
\Delta w_x- (\CP w_{xx}+v_xw_x+vw_{xx}+w_xw_y+ww_{xy})=q_{yx} & \mbox{in }S,
\end{array}\right.
\end{equation}
which, after subtracting the two equations and by the Schwarz Theorem, gives
\begin{equation}\label{nopressure}
\Delta(v_y-w_x)- \big(\CP(v_{xy}-w_{xx})+\CP''w-v\Delta w+w\Delta v\big)=0\quad\mbox{in }S\, .
\end{equation}
The linear part of \eqref{nopressure} reads
$$
\Delta(v_y-w_x)- \big[\CP(y)(v_{xy}-w_{xx})+6Aw\big]=0\quad\mbox{in }S\, .
$$
We  associate to $(v,w)$ a (scalar) stream function $\psi\in H^2_0(S)$ (see e.g.\ \cite[Theorem 2(i)]{rabier2})
satisfying $v=\psi_y$ and $w=-\psi_x$ so that the last equation becomes
\begin{equation}\label{biha}
\Delta^2\psi- \big[\CP(y)(\psi_{xyy}+\psi_{xxx})-6A\psi_x\big]=0\quad\mbox{in }S\, ,
\end{equation}
while the boundary conditions \eqref{nonlinear}$_4$ together with the zero flux condition \eqref{nonlinear}$_5$ imply that
\begin{equation}\label{boundarystream}
\psi(x,\pm1)=\psi_x(x,\pm1)=\psi_y(x,\pm1)=0\, .
\end{equation}
We shall be able to prove the invertibility of the operator $\mathcal{L}$ in suitable spaces under the assumption that\footnote{Recall that the assumption $A\le 0$ is made without loss of generality. \label{foot:p}}
\begin{equation}\label{ABC}
	(A,B,C)\neq(0,0,0),\qquad A\le0,\qquad|B|\le3A+C,
\end{equation}
which implies, in particular, that $2A+B+2C>0$. Moreover, from \eqref{ABC} we have
\begin{equation}\label{parabolasign}
	\CP(y)\doteq3Ay^2+By+C\ge0\quad\forall y\in[-1,1]\, ,\qquad \CP(y)>0\quad\forall y\in(-1,1).
\end{equation}
\begin{remark}The requirement \eqref{parabolasign} on $\CP$ excludes the phenomenon of flow reversal, and it is crucial to obtain, in particular, the inequalities \eqref{A00} and \eqref{A0} below, which, in turn, are instrumental to the proof of injectivity of $\mathcal L$. In Section \ref{sec:6} we shall show, by a counter-example,  that injectivity can be lost in presence of flow reversal.
\end{remark}
\begin{remark}In the recent paper \cite{GKRS}, existence of solutions in a 2D domain with multiple outlets to infinity is established in the Ladyzhenskaya \& Solonnikov class also in presence of sources and sinks, and with non-homogeneous boundary conditions along the outlet walls. It is worth observing that for the validity of this result, condition \eqref{ABC} is not needed.
\end{remark}
\section{A non-homogeneous  Orr-Sommerfeld equation}\label{sec:3}

Referring back to the considerations discussed in Section \ref{sec:2}, we want now to demonstrate that the only solution to the system \eqref{biha}-\eqref{boundarystream} (within a certain class of functions that we will specify later) is the trivial one. In order to do so, we take the partial Fourier transform of the stream function $\psi$, namely,
$$
\phi_\xi(y)\doteq\int_\R e^{-i\xi x}\psi(x,y)\, dx\,,
$$
to formally deduce, from the equation \eqref{biha} (with inhomogeneous right hand side) and the boundary conditions \eqref{boundarystream}, the following fourth-order ODE boundary value problem:
\begin{equation}\label{OS}
	\left\{\begin{array}{ll}
		\phi_\xi''''(y)-2\xi^2\phi_\xi''(y)+\xi^4\phi_\xi(y)-i\xi  \Big[\CP(y)(\phi_\xi''(y)-\xi^2\phi_\xi(y))-6A\phi_\xi(y)\Big]=h_\xi(y) \\[3pt]
		\phi_\xi(\pm1)=\phi_\xi'(\pm1)=0\, .
	\end{array}\right.
\end{equation}
which is the non-homogeneous   {\em Orr-Sommerfeld equation}  (corresponding to the eigenvalue 0).

We first present the unique solvability of \eqref{OS} for any $h_\xi \in H^{-1}((-1,1); \mathbb{C})$. Recall that $H^{-1}((-1,1); \mathbb{C})$ is defined as the dual space of $H^1_0((-1,1); \mathbb{C})$. In the sequel we shall often write $X \lesssim Y$ for $X \le c Y$ where $c>0$ is a constant that may depend only on $A, B, C$.

\begin{proposition} \label{prop:OS}
	Assume \eqref{ABC}, and let $\CP$ be as in \eqref{parabolasign}. Then for any  $\xi \neq 0$ and $h_\xi \in H^{-1}((-1,1);\mathbb{C})$, there exists a unique solution $\phi_\xi \in H^3((-1,1); \mathbb{C})\cap H^2_0((-1,1); \mathbb{C})$ to \eqref{OS} satisfying
	 \begin{equation} \label{key-apriori-0}
	 	\int_{-1}^{1} (  |\phi_\xi''(y)|^2 + 2\xi^2 |\phi_\xi'(y)|^2 + \xi^4 |\phi_\xi(y)|^2) dy \lesssim  \min\left\{\|h_\xi\|_{H^{-1}((-1,1);\mathbb{C})}^2, \xi^{-2}\|h_\xi\|^2_{L^2((-1,1);\mathbb{C})} \right\}.
	 \end{equation}
\end{proposition}

\begin{proof}
	We generalize the ideas introduced by McLeod \cite{mcleod} when $\CP=\PP$ (Poiseuille flow), see also \cite[Sections 10.4 and 10.5]{HMC}. For notational simplicity, we omit the dependence of $\phi_\xi, h_\xi$ on $\xi$ and simply write $\phi, h$.
	
	For now, suppose that $h \in C^\infty([-1,1];\mathbb{C})$ so that any solution $\phi$ to \eqref{OS} is also smooth up to the boundary, and the main goal is to derive a priori estimates. We integrate the equation in \eqref{OS} over the interval $(-1,y)$ to obtain
	$$
	\phi'''(y)-2\xi^2\phi'(y)+\xi^4\int_{-1}^{y}\phi(t)dt-i\xi \, \int_{-1}^{y}\Big[\CP(t)(\phi''(t)-\xi^2\phi(t))-6A\phi(t)\Big]dt= \phi'''(0) + \int_{-1}^y h(t) dt\, .
	$$
	After two integration by parts, this equation becomes
	\begin{align}\label{beforesigma}
		\phi'''(y)-2\xi^2\phi'(y)+\xi^4\int_{-1}^{y}\phi(t)dt
		&-i\xi \left[\CP(y)\phi'(y)-\CP'(y)\phi(y)-\xi^2\int_{-1}^{y}\CP(t)\phi(t)dt\right] \nonumber \\
		&=c + \int_{-1}^y h(t) dt\, .
	\end{align}
	
	Now, we introduce the auxiliary function
	\begin{equation}\label{defsigmaphi}
		\sigma(y)\doteq\frac{\phi(y)}{\CP(y)}\ \Longrightarrow\ \phi(y)=\CP(y)\sigma(y).
	\end{equation}
	For any $A,B,C\in\R$ satisfying \eqref{ABC}, $\CP(y)$ annihilates of at most order $1$ at $y=\pm1$, see \eqref{parabolasign}.
	Combined with \eqref{OS}$_2$, this shows that
	\begin{equation}\label{newbc}
		\sigma(\pm1)=0
	\end{equation}
	and, by differentiating \eqref{defsigmaphi} with respect to $y$ and evaluating at $y=\pm1$, from \eqref{newbc} and \eqref{OS}$_2$ we deduce
	\begin{equation}\label{boundarypm1}
		\CP(\pm1)\sigma'(\pm1)=\CP'(\pm1)\sigma(\pm1)+\CP(\pm1)\sigma'(\pm1)=\phi'(\pm1)=0.
	\end{equation}
	From this, we derive one more condition at $y=\pm1$. By \eqref{ABC} we infer that $C>0$ and three cases may occur:
	$$
	\left\{\begin{array}{lll}
		6A\le B\le-6A\ \Rightarrow\ \CP'(1)\le0,\ \CP'(-1)\ge0,\\
		\CP'(1)=B+6A>0\ \Rightarrow\ \CP(1)=3A+B+C>0\ \stackrel{\eqref{boundarypm1}}{\Rightarrow}\ \sigma'(1)=0,\\
		\CP'(-1)=B-6A<0\ \Rightarrow\ \CP(-1)=3A-B+C>0\ \stackrel{\eqref{boundarypm1}}{\Rightarrow}\ \sigma'(-1)=0.
	\end{array}\right.
	$$
	In any of these three (exhaustive) cases, it happens that
	\begin{equation}\label{A00}
		\CP'(1)|\sigma'(1)|^2-\CP'(-1)|\sigma'(-1)|^2\le0.
	\end{equation}
	
	Moreover, \eqref{defsigmaphi} transforms \eqref{beforesigma} into
	$$
	\CP(y)\sigma'''(y)+3\CP'(y)\sigma''(y)+18A\sigma'(y)-2\xi^2\big[\CP(y)\sigma'(y)+\CP'(y)\sigma(y)\big]+\xi^4\int_{-1}^{y}\CP(t)\sigma(t)dt
	$$
	$$
	-i\xi \, \left[\CP(y)^2\sigma'(y)-\xi^2\int_{-1}^{y}\CP(t)^2\sigma(t)dt\right]=\phi'''(0) + \int_{-1}^y h(t) dt\, .
	$$
	Multiply this equation by the conjugate $\overline{\sigma}'(y)$ and integrate over $(-1,1)$ to obtain:
	\begin{align}
		\int_{-1}^{1}\big[\CP(y)\sigma'''(y)+3\CP'(y)\sigma''(y)+18A\sigma'(y)\big]\overline{\sigma}'(y)dy& \label{first}\\
		-2\xi^2\int_{-1}^{1}\CP(y)|\sigma'(y)|^2dy-2\xi^2\int_{-1}^{1}\CP'(y)\sigma(y)\overline{\sigma}'(y)dy
		+\xi^4\int_{-1}^{1}\int_{-1}^{y}\CP(t)\sigma(t)dt\, \overline{\sigma}'(y)dy \label{second}&\\
		-i\xi \, \left[\int_{-1}^{1}\CP(y)^2|\sigma'(y)|^2dy-\xi^2\int_{-1}^{1}\int_{-1}^{y}\CP(t)^2\sigma(t)dt\, \overline{\sigma}'(y)
		dy\right]  &= - \int_{-1}^1 h(y) \bar{\sigma}(y) dy\, .\label{third}
	\end{align}
	Note that for the right hand side we have used integration by parts and \eqref{newbc}.
	
	For the first term in \eqref{first}, through an integration by parts we obtain
	$$
	\int_{-1}^{1}\CP(y)\sigma'''(y)\overline{\sigma}'(y)dy=-\int_{-1}^{1}\CP(y)|\sigma''(y)|^2dy-\int_{-1}^{1}\CP'(y)\sigma''(y)\overline{\sigma}'(y)dy
	$$
	so that the three terms in \eqref{first} give
	$$
	\int_{-1}^{1}\big[\CP(y)\sigma'''(y)+3\CP'(y)\sigma''(y)+18A\sigma'(y)\big]\overline{\sigma}'(y)dy
	$$
	\begin{equation}\label{3terms}
		=-\int_{-1}^{1}\CP(y)|\sigma''(y)|^2dy+2\int_{-1}^{1}\CP'(y)\sigma''(y)\overline{\sigma}'(y)dy+18A\int_{-1}^{1}|\sigma'(y)|^2dy.
	\end{equation}
	By separating the real and imaginary part of $\sigma$, say $\sigma(y)=\sigma_1(y)+i\sigma_2(y)$ with $\sigma_1,\sigma_2\in\R$, we see that
	$$\sigma''(y)\overline{\sigma}'(y)=\big(\sigma_1''(y)+i\sigma_2''(y)\big)\big(\sigma_1'(y)-i\sigma_2'(y)\big)$$
	and, therefore, its real part is $\sigma_1'(y)\sigma_1''(y)+\sigma_2'(y)\sigma_2''(y)$. Thus the real part of the second integral in \eqref{3terms} is
	\begin{eqnarray}
		\mbox{Re}\left[2\int_{-1}^{1}\CP'(y)\sigma''(y)\overline{\sigma}'(y)dy\right] &=& 2\int_{-1}^{1}\CP'(y)\big(\sigma_1'(y)\sigma_1''(y)+\sigma_2'(y)\sigma_2''(y)\big)dy \notag\\
		&=& \int_{-1}^{1}\CP'(y)\big(\sigma_1'(y)^2+\sigma_2'(y)^2\big)'dy \notag\\
		\mbox{(by parts) }&=& -6A\int_{-1}^{1}|\sigma'(y)|^2dy+\big[\CP'(y)\, |\sigma'(y)|^2\big]_{-1}^1\label{as}
	\end{eqnarray}
	which, combined with \eqref{A00} and \eqref{3terms}, proves that
	\begin{equation}\label{Refirst}
		\begin{aligned}
			& \quad \mbox{Re}\left[\int_{-1}^{1}\big[\CP(y)\sigma'''(y)+3\CP'(y)\sigma''(y)+18A\sigma'(y)\big]\overline{\sigma}'(y)dy\right] \\[4pt]
			& \le 12A\int_{-1}^{1}|\sigma'(y)|^2dy-\int_{-1}^{1}\CP(y)|\sigma''(y)|^2dy\,.
		\end{aligned}
	\end{equation}
	
	Next, we analyse the second term in \eqref{second}. By arguing as for \eqref{as}, we obtain
	$$
	\mbox{Re}\left[-2\xi^2\int_{-1}^{1}\CP'(y)\sigma(y)\overline{\sigma}'(y)dy\right]=6A\xi^2\int_{-1}^{1}|\sigma(y)|^2dy
	$$
	in view of \eqref{newbc}. Then we integrate by parts the third term in \eqref{second} to get
	$$
	\xi^4\int_{-1}^{1}\int_{-1}^{y}\CP(t)\sigma(t)dt\, \overline{\sigma}'(y)dy=-\xi^4\int_{-1}^{1}\CP(y)|\sigma(y)|^2dt\, ,
	$$
	for which we used again \eqref{newbc}. Collecting terms, we have the following expression for the real part of \eqref{second}:
	\begin{eqnarray}
		& & \mbox{Re}\left[-2\xi^2\int_{-1}^{1}\CP(y)|\sigma'(y)|^2dy-2\xi^2\int_{-1}^{1}\CP'(y)\sigma(y)\overline{\sigma}'(y)dy
		+\xi^4\int_{-1}^{1}\int_{-1}^{y}\CP(t)\sigma(t)dt\, \overline{\sigma}'(y)dy\right] \notag \\
		&=& -2\xi^2\int_{-1}^{1}\CP(y)|\sigma'(y)|^2dy+6A\xi^2\int_{-1}^{1}|\sigma(t)|^2dt-\xi^4\int_{-1}^{1}\CP(y)|\sigma(y)|^2dy. \label{Resecond}
	\end{eqnarray}
	
	Finally, with an integration by parts, we see that the term inside brackets in \eqref{third} equals
	$$
	\int_{-1}^{1}\CP(y)^2|\sigma'(y)|^2dy-\xi^2\int_{-1}^{1}\int_{-1}^{y}\CP(t)^2\sigma(t)dt\, \overline{\sigma}'(y)dy=
	\int_{-1}^{1}\!\CP(y)^2|\sigma'(y)|^2dy+\xi^2\!\int_{-1}^{1}\!\CP(y)^2|\sigma(y)|^2dy
	$$
	and, hence, is a real number. Therefore, the first term of \eqref{third} is an imaginary number with null real part. Together with \eqref{Refirst} and
	\eqref{Resecond}, by taking the real part of \eqref{first}+\eqref{second}+\eqref{third} we get
	\begin{align} \label{key-apriori}
	 &\quad -12A\int_{-1}^{1} |\sigma'(y)|^2 dy - 6A \xi^2 \int_{-1}^1 |\sigma(y)|^2dy + \int_{-1}^{1}\CP(y)(|\sigma''(y)|^2+2\xi^2 |\sigma'(y)|^2 + \xi^4 |\sigma(y)|^2)dy  \nonumber \\
	 &\le \mbox{Re} \left[\int_{-1}^1 h(y) \bar{\sigma}(y) dy\right] \lesssim \|h\|_{H^{-1}((-1,1);\mathbb{C})} \|\sigma'\|_{L^2((-1,1);\mathbb{C})} \, .
	\end{align}
	Using Lemma \ref{lem:poincare} (proved below), we deduce from \eqref{key-apriori}  the key a priori estimate
	\begin{equation} \label{key-apriori-2}
		\int_{-1}^{1} ( |\sigma'(y)|^2 + \xi^2 |\sigma(y)|^2 + |\phi''(y)|^2 + 2\xi^2 |\phi'(y)|^2 + \xi^4 |\phi(y)|^2) dy \lesssim \|h\|_{H^{-1}((-1,1);\mathbb{C})}^2.
	\end{equation}

	By \eqref{key-apriori-2}, we know that  $\phi \equiv 0$ if $h \equiv 0$, which shows the injectivity of $\mathcal{B}$ from  $ H^3((-1,1); \mathbb{C})\cap H^2_0((-1,1); \mathbb{C})$ to $H^{-1}((-1,1); \mathbb{C})$,
where $\mathcal{B}$ is a linear operator generated by the left hand side of~(\ref{OS})$_1$. 	
On the other hand, the corresponding surjectivity follows from the injectivity and the method of constants variation, exactly as in \cite[Section 10.4]{HMC}.
	
	In the case that $h \in L^2((-1,1);\mathbb{C})$, instead of last estimate in \eqref{key-apriori} we use the Young inequality:
	\begin{align}
		\mbox{Re} \left[\int_{-1}^1 h(y) \bar{\sigma}(y) dy\right] &\le \|h\|_{L^2((-1,1);\mathbb{C})} \|\sigma\|_{L^2((-1,1);\mathbb{C})} \nonumber \\
		&\le \dfrac{1}{\varepsilon} \xi^{-2} \|h\|_{L^2((-1,1);\mathbb{C})}^2 + \dfrac{\varepsilon}{4} \xi^2 \|\sigma\|_{L^2((-1,1);\mathbb{C})}^2 \,,
	\end{align}
for a sufficiently small $\varepsilon>0$ (suitable to apply~(\ref{poincare}) from Lemma~\ref{lem:poincare}),	which leads to the desired estimate \eqref{key-apriori-0}.
	Thus we have proved all the claims of the Proposition.
\end{proof}

\begin{lemma} \label{lem:poincare}
	Under our assumption \eqref{ABC} on $A,B,C$ and the above notations, there holds
	\begin{align}
		& -12A\int_{-1}^{1} |\sigma'(y)|^2 dy - 6A \xi^2 \int_{-1}^1 |\sigma(y)|^2dy + \int_{-1}^{1}\CP(y)(|\sigma''(y)|^2+2\xi^2 |\sigma'(y)|^2 + \xi^4 |\sigma(y)|^2)dy \label{sigma-energy} \\
		&\quad \quad \gtrsim \int_{-1}^{1} \left( |\sigma'(y)|^2 + \xi^2 |\sigma(y)|^2 + |\phi''(y)|^2 + 2\xi^2 |\phi'(y)|^2 +\xi^4 |\phi(y)|^2 \right) dy\,. \label{poincare}
	\end{align}
\end{lemma}

\begin{proof}
	If $A<0$, then the first two terms concerning $\sigma', \sigma$ in \eqref{poincare} are directly controlled by the first integral in \eqref{sigma-energy}. By \eqref{newbc} we have the Poincar\'e inequality
	\begin{equation} \label{poin}
		\int_{-1}^1 |\sigma'(y)|^2 dy \ge \frac{\pi^2}{4} \int_{-1}^1 |\sigma(y)|^2 dy.
	\end{equation}
	Using \eqref{poin} and
	\begin{equation} \label{phisigma}
		\phi''(y) = \mathcal{F}(y) \sigma''(y) + 2 \mathcal{F}'(y) \sigma'(y) + \mathcal{F}''(y) \sigma(y), \quad \phi'(y) = \mathcal{F}(y) \sigma'(y) + \mathcal{F}'(y) \sigma(y),
	\end{equation}
	it is easy to see that the three terms concerning $\phi'', \phi', \phi$ in  \eqref{poincare} are also controlled by \eqref{sigma-energy}.
	
	 Now, it remains to study the case $A = 0$. By \eqref{ABC}, either $\mathcal{F}(1)$ or $\mathcal{F}(-1)$ is positive. Without loss of generality, we assume that $\mathcal{F}(-1)>0$. Notice that
	\begin{align} \label{sigma-basic}
		|\sigma'(y) -\sigma'(-1)| =  \left|\int_{-1}^y \sigma''(t) dt  \right|
		&\le  \left(\int_{-1}^y \frac{1}{\mathcal{F}(t)} dt \right)^\frac12 \left(\int_{-1}^y \mathcal{F}(t) |\sigma''(t)|^2 dt \right)^\frac12 \nonumber \\
		&\lesssim (1+|\log (1-y)|)^\frac12  \left(\int_{-1}^1 \mathcal{F}(t) |\sigma''(t)|^2 dt \right)^\frac12,
	\end{align}
	which then implies
	\begin{equation} \label{910a}
		\int_{-1}^1 |\sigma'(y)|^2 dy \lesssim  |\sigma'(-1)|^2 +  \int_{-1}^1 \mathcal{F}(y) |\sigma''(y)|^2 dy.
	\end{equation}
	Moreover, from \eqref{newbc} and \eqref{sigma-basic} we know that
	\begin{equation} \label{910b}
		2 |\sigma'(-1)| = \left|\int_{-1}^1 \left( \sigma'(-1) -\sigma'(y) \right)  dy \right| \lesssim \left(\int_{-1}^1 \mathcal{F}(y) |\sigma''(y)|^2 dy \right)^\frac12.
	\end{equation}
	Hence, using \eqref{910a} and \eqref{910b} we get
	\begin{equation} \label{911a}
		\int_{-1}^1 |\sigma'(y)|^2 dy \lesssim \int_{-1}^1 \mathcal{F}(y) |\sigma''(y)|^2 dy .
	\end{equation}
	Similarly, we have
	\begin{align} \label{sigma-basic-2}
		|\sigma(y)|  =  \left|\int_{-1}^y \sigma'(t) dt  \right|
		&\le  \left(\int_{-1}^y \frac{1}{\mathcal{F}(t)} dt \right)^\frac12 \left(\int_{-1}^y \mathcal{F}(t) |\sigma'(t)|^2 dt \right)^\frac12 \nonumber \\
		&\lesssim  (1+|\log (1-y)|)^\frac12  \left(\int_{-1}^1 \mathcal{F}(t) |\sigma'(t)|^2 dt \right)^\frac12,
	\end{align}
	which then implies
	\begin{equation} \label{911b}
		\int_{-1}^1 |\sigma(y)|^2 dy \lesssim \int_{-1}^1 \mathcal{F}(y) |\sigma'(y)|^2 dy.
	\end{equation}
	Combining \eqref{phisigma}, \eqref{911a}, and \eqref{911b} we finish the proof of the lemma.
\end{proof}

\section{Uniqueness  and structural stability in the $L^2$-setting} \label{nobranch1}

Using Proposition \ref{prop:OS}, we will now deduce the existence and uniqueness of the solution for the linearization of the Couette-Poiseuille problem with external force in $L^2(S)$, and we will also show its nonlinear analogue for ``small" forces. An important consequence of the latter is that the Couette-Poiseuille flow is unique in the class \eqref{fl} with $\|\Q\|_{H^2(S)}$ sufficiently ``small."

\begin{theorem} \label{L2}
	Assume \eqref{ABC}. Then, for any $\f \in L^2(S)$, the linear problem
		\begin{equation} \label{linearproblem}
			\left\{
			\begin{aligned}
				&-\Delta\Q+ (\Q\cdot\nabla)\U_{*} +  (\U_{*}\cdot\nabla)\Q +\nabla q = \f \quad \mbox{in }S,\\[6pt]
&\nabla \cdot \Q = 0 \quad \mbox{in }S, \\[6pt]
				&\Q(x, \pm 1) = 0 \quad \forall x \in \R, \\[6pt]
				&\int_{-1}^1 \Q(0,y) \cdot \e_1 dy = 0,
			\end{aligned} \right.
		\end{equation}
		has a unique solution $(\Q,q)\in \mathcal H(S)$.
		Moreover, 		\begin{equation} \label{Qbound}
			\|\Q\|_{H^2(S)} +\|\nabla q\|_{L^2(S)}\le \kappa_0\, \|\f\|_{L^2(S)}
		\end{equation}
		for some positive constant $\kappa_0$ depending  on $A,B$ and $C$.
\end{theorem}
\begin{proof}
The idea is to perform the Fourier transform in $x$ and to solve each Fourier mode using the result of Proposition \ref{prop:OS}.
	
	\smallskip
	
	Let $(f, g)$ be the components of the external force $\f$. Similar to the discussion in Section \ref{sec:2}, the linear problem \eqref{linearproblem} can be rewritten as
	\begin{equation}\label{biha-f}
		\Delta^2\psi- \big[\CP(y)(\psi_{xyy}+\psi_{xxx})-6A\psi_x\big]=g_x - f_y \quad \mbox{in }S.
	\end{equation}
	We make a Fourier transform in $x$-variable and obtain
	\begin{equation}\label{OS-2}
		\left\{\begin{array}{ll}
			\phi_\xi''''(y)-2\xi^2\phi_\xi''(y)+\xi^4\phi_\xi(y)-i\xi \, \Big[\CP(y)(\phi_\xi''(y)-\xi^2\phi_\xi(y))-6A\phi_\xi(y)\Big]=i\xi \, \widetilde{g}_\xi(y) -\widetilde{f}_\xi'(y) \, , \\
			\phi_\xi(\pm1)=\phi_\xi'(\pm1)=0\, ,
		\end{array}\right.
	\end{equation}
	where $\widetilde{f}_\xi, \widetilde{g}_\xi$ denote the Fourier transforms of $f$ and $g$ in $x$ respectively. For any $\xi \in \R$, the existence and uniqueness of $\phi_\xi \in H^3 \cap H^1_0((-1,1); \mathbb{C})$ is guaranteed by Proposition  \ref{prop:OS}. Moreover, we have the uniform-in-$\xi$ estimates
	\begin{equation} \label{key-apriori-00}
		\int_{-1}^{1} (  |\phi''(y)|^2 + 2\xi^2 |\phi'(y)|^2 + \xi^4 |\phi(y)|^2) dy \lesssim (\|\widetilde{g}_\xi\|_{L^2((-1,1); \mathbb{C})} + \|\widetilde{f}_\xi\|_{L^2((-1,1); \mathbb{C})})\,.
	\end{equation}
	Integrating \eqref{key-apriori-00} over $\xi \in \R$ we get
	\begin{equation}
		\|\Q\|_{H^1(S)} \lesssim \|\f\|_{L^2(S)}.
	\end{equation}
	The bound for $\|\Q\|_{H^2(S)}$ then follows from Stokes estimates (see, for instance, \cite[Lemma VI.1.2]{G}). Indeed, from \eqref{linearproblem}, for any $a \in \mathbb{R}$ we have
	\begin{align} \label{stokes1}
		\|\Q\|_{H^2((a,a+1)\times (-1,1))} &\lesssim \|\f\|_{L^2((a-1,a+2)\times (-1,1))} + \|\U_* \cdot \nabla \Q\|_{L^2((a-1,a+2)\times (-1,1))} \nonumber \\
		&\quad + \|\Q \cdot \nabla \U_*\|_{L^2((a-1,a+2)\times (-1,1))} + \|\Q\|_{H^1((a-1,a+2)\times (-1,1))} \nonumber \\
		&\lesssim \|\f\|_{L^2((a-1,a+2)\times (-1,1))} +   \|\Q\|_{H^1((a-1,a+2)\times (-1,1))}.
	\end{align}
	Taking the square of the above estimate and summing over $a \in \mathbb{Z}$, we get the desired estimate \eqref{Qbound}.
\end{proof}

Employing Theorem \ref{L2} in conjunction with the Contraction-Mapping Theorem, we are able to prove its nonlinear counterpart in the following theorem.
\begin{theorem} \label{cor1}
Assume \eqref{ABC}. Then, there exists $\eta>0$ depending only on $A,B,C$ such that, for any $\f \in L^2(S)$ with $\|\f\|_{L^2(S)} \le \eta$, there is a corresponding unique solution   $(\U,p) \in H^2_{\text{loc}}(S)\times H^{1}_{\text{loc}}(S)$ to the nonlinear problem
\begin{equation}\label{nonlinearproblem}
	\left\{\begin{array}{ll}
		-\Delta\U+ (\U\cdot\nabla)\U+\nabla p=\f \quad \mbox{in }S, \\[6pt]
		\nabla\cdot\U=0\quad\mbox{in }S,\\[6pt]
		\U=(3A-B+C) \e_1 \mbox{ for }y=-1\, ,\quad \U=((3A+B+C)) \e_1 \mbox{ for }y=1,\\[6pt]
		\int_{-1}^1 \U(0,y) \cdot \e_1 dy = \Phi\,. 
	\end{array}\right.
\end{equation}
This solution is of the form \eqref{form} where $(\Q,q)\in \mathcal H(S)$ satisfies
\begin{equation}\label{EsT}
\|\Q\|_{H^2(S)}+\|\nabla q\|_{L^2(S)} \le \kappa\, \|\f\|_{L^2(S)}\,,
\end{equation}
with $\kappa>0$ depending on $A,B$ and $C$.
\end{theorem}
\begin{proof}
We use the contraction-mapping theorem combined with the linear estimate \eqref{Qbound}. To this end, for $\delta>0$ let
$$
H_{(\delta)}^2(S)\,\dot=\,\{(\Q,q)\in H^2(S):\ \|\Q\|_{H^2(S)}\le\delta\}
$$
and consider the map
$$
{\sf M}: \mathbf{w}\in H_{(\delta)}^2(S) \mapsto \Q\in  H^2(S)
$$
where $\Q$ solves the following problem
		\begin{equation} \label{CP}
			\left\{
			\begin{aligned}
				&-\Delta\Q+ (\Q\cdot\nabla)\U_{*} +  (\U_{*}\cdot\nabla)\Q +\nabla q =-(\mathbf{w}\cdot\nabla)\mathbf{w}+ \f\,\dot=\,\mathbf{F} \quad \mbox{in }S, \\[3pt] 
				& \nabla \cdot \Q = 0 \quad \mbox{in }S, \\[3pt]
				&\Q(x, \pm 1) = 0, \quad \forall x \in \R, \\[3pt]
				&\int_{-1}^1 \Q(0,y)\cdot\e_1 dy = 0.
			\end{aligned} \right.
		\end{equation}

	\smallskip\noindent
By assumption on $\f$ and the embedding inequality
\begin{equation}\label{em}
	\|(\mathbf{u} \cdot \nabla) \mathbf{w} \|_{L^2(S)} \le c_1\,\|\mathbf{u}\|_{H^2(S)} \|\mathbf{w}\|_{H^2(S)},\ \ \mathbf{u}\,, \mathbf{w}\in H^2(S)\,,
\end{equation}
we infer $\mathbf{F}\in L^2(S)$, so that with the help of Theorem \ref{L2}, we deduce, on the one hand, that {\sf M} is well defined and, on the other hand, that the solution $\Q$ to \eqref{CP} satisfies
\begin{equation}\label{dl0}
\|\Q\|_{H^2(S)}\le \kappa_0\, (c_1\,\delta^2+\|\mathbf{f}\|_{L^2(S)})\,.
\end{equation}
Choosing
\begin{equation}\label{dl}
(4\kappa_0c_1)^{-1}\ge \delta\ge4\kappa_0\,\|\f\|_{L^2(S)}\,,
\end{equation}
from \eqref{dl0} we conclude
\begin{equation}\label{dl1}
\|\Q\|_{H^2(S)}\le \mbox{$\frac12$}\,\delta
\end{equation}
which shows that ${\sf M}$ is a self-map. Further, setting $\Q\,\dot=\,{\sf M}(\mathbf{w}_1)-{\sf M}(\mathbf{w}_2)$, $\mathbf{w}\,\dot=\mathbf{w}_1-\mathbf{w}_2$, from \eqref{CP} we get
		$$
			\left\{
			\begin{aligned}
				&-\Delta\Q+ (\Q\cdot\nabla)\U_{*} +  (\U_{*}\cdot\nabla)\Q +\nabla q =-(\mathbf{w}_1\cdot\nabla)\mathbf{w}-(\mathbf{w}\cdot\nabla)\mathbf{w}_2 \quad \mbox{in }S, \\[3pt] 
				& \nabla \cdot \Q = 0 \quad \mbox{in }S, \\[3pt] 
				&\Q(x, \pm 1) = 0, \quad \forall x \in \R, \\[3pt]
				&\int_{-1}^1 \Q(0,y)\cdot\e_1 dy = 0.
			\end{aligned} \right.
		$$
Employing Theorem \ref{L2} in this problem along with \eqref{em} and \eqref{dl}, we show
$$
\|\Q\|_{H^2(S)}\le \kappa_0c_1(\|\mathbf{w}_1\|_{H^2(S)}+\|\mathbf{w}_2\|_{H^2(S)})\|\mathbf{w}\|_{H^2(S)}\le 2\kappa_0c_1\delta\,\|\mathbf{w}\|_{H^2(S)}\le \mbox{$\frac12$}\,\|\mathbf{w}\|_{H^2(S)}\,,
$$
which proves that ${\sf M}$ is a contraction and, therefore, the desired existence and uniqueness result. Moreover, the estimate \eqref{EsT} follows from \eqref{dl0}, \eqref{dl} and \eqref{form}$_1$.
\end{proof}
Theorem \ref{cor1} furnishes the following immediate but important consequence which establishes the uniqueness of Couette--Poiseuille flow in a suitable class of solutions.
\begin{corollary}\label{Cor1} Let the assumption of Theorem \ref{cor1} hold. Then, there exists $\delta_0>0$, depending on $A,B,C$, such that the only solution to \eqref{nonlinearproblem} with $\mathbf{f}\equiv\0$ of the type
$$
\mathbf{u}=\mathbf{u}_*+\Q\,,\ \ p=p_*+q
$$
where $(\Q,q)\in \mathcal H(S)$ with  $\|\Q\|_{H^1(S)}\le\delta_0$, is the Couette--Poiseuille flow $\mathbf{u}=\mathbf{u}_*\,, p=p_*$\,.
\end{corollary}
\begin{proof} From the proof of Theorem \ref{cor1} it follows that the stated property is true if
\begin{equation}\label{ch}
\|\Q\|_{H^2(S)}\le \delta\,,
\end{equation}
with $\delta$ as in \eqref{dl}.
We shall show that to obtain the latter it is enough to assume suitable smallness of $\Q$ in the $H^1$ norm. In fact, from the local Stokes estimates \cite[Lemma VI.1.2]{G} applied to \eqref{poiseuille2dqq}$_{1,2,3}$, we show
\begin{eqnarray}
		\|\Q\|_{H^2((a,a+1)\times (-1,1))} &\!\!\!\lesssim&  \|\Q \cdot \nabla \Q\|_{L^2((a-1,a+2)\times (-1,1))} +   \|\Q\|_{H^1((a-1,a+2)\times (-1,1))} \label{StEs}\\
		&\!\!\!\lesssim&   \|\Q\|_{H^1((a-1,a+2)\times (-1,1))} \|\Q\|_{H^2((a-1,a+2)\times (-1,1))} +  \| \Q\|_{H^1((a-1,a+2)\times (-1,1))}.\notag
\end{eqnarray}
       Taking the square of the above estimate and summing over $a \in \mathbb{Z}$, and assuming $\|\Q\|_{H^1(S)}$ below a suitable constant, we get
	 \begin{align*}
	 	\|\Q\|_{H^2(S)} &\lesssim  \|\Q\|_{H^1(S)}.
	 \end{align*}
	 which implies the claimed property \eqref{ch}.
\end{proof}
\section{Uniqueness and structural stability in the locally $L^2$-setting}\label{sec:5}
The goal of this section is to demonstrate the analogues of the theorems obtained in the previous section, but now in a local context. To this end, we begin to show the following
 result.
\begin{theorem}\label{bounded}
Assuming \eqref{ABC}, for any $\f \in \X^0$ the linear problem \eqref{linearproblem} has a unique solution $(\Q,q) \in \mathbb H(S)$. Moreover, there exists $\gamma_0>0$ depending on $A,B,C$ such that
	\begin{equation} \label{X20}
		\|\Q\|_{\X^2}+\|\nabla q\|_{\X^0} \le \gamma_{0} \|\f\|_{\X^0}\,.
	\end{equation}
	
\end{theorem}

\begin{proof}
Here, we have to extend the proof of Proposition \ref{prop:OS} to a two-dimensional PDE setting. The argument hinges on a  careful analysis of boundary terms arising from integration by parts and a dichotomy governing the asymptotic behavior of solutions.
	
		\medskip
	
	\underline{Step 1: Introducing the auxiliary function $\sigma$.}

	\smallskip
	Suppose for now that $\f$ and $\Q$ are smooth vector fields on $\bar{S}$, and our first goal is to derive effective a priori estimates. By integrating \eqref{biha-f} over $(-1,y)$, and recalling that $\psi=|\nabla\psi|=0$ on $\partial S$, we obtain
	\begin{equation} \label{bihaint-f}
	\begin{aligned}
		& \psi_{yyy}(x,y) - \psi_{yyy}(x,-1)+2\psi_{xxy}(x,y)+\int_{-1}^{y}\!\!\psi_{xxxx}(x,t)dt \\
		& -\int_{-1}^{y}\!\!\big[\CP(t)(\psi_{xyy}(x,t)+\psi_{xxx}(x,t))
		-6A\psi_x(x,t)\big]dt = f(x,-1) - f(x,y) + \int_{-1}^y g_x(x,t) dt
	\end{aligned}
\end{equation}
	for each $x\in\R$. Let
	\begin{equation}\label{defsigma}
		\sigma(x,y)\doteq\frac{\psi(x,y)}{\CP(y)}\ \Longrightarrow\ \psi(x,y)=\CP(y)\sigma(x,y).
	\end{equation}
	Using again $\psi=|\nabla\psi|=0$ on $\partial S$, for any $A,B,C\in\R$ satisfying \eqref{ABC}, $\CP(y)$ annihilates of at most order $1$ at $y=\pm1$, see \eqref{parabolasign},  we have $\sigma\in C^\infty(\bar S)$ with
	\begin{equation}\label{sigmapm1}
		\sigma(x,\pm1)=\sigma_{x}(x,\pm1)=\sigma_{xx}(x,\pm1)=0\qquad\forall x\in\R
	\end{equation}
	and, by differentiating \eqref{defsigma} with respect to $y$ and evaluating at $y=\pm1$, we infer
	\begin{equation}\label{boundary}
		\CP(\pm1)\sigma_y(x,\pm1)=\CP'(\pm1)\sigma(x,\pm1)+\CP(\pm1)\sigma_y(x,\pm1)=\psi_y(x,\pm1)=0\qquad\forall x\in\R.
	\end{equation}
	We derive one more condition at $y=\pm1$. By \eqref{ABC} we infer that $C>0$ and three cases may occur:
	$$
	\left\{\begin{array}{lll}
		6A\le B\le-6A\ \Rightarrow\ \CP'(1)\le0,\ \CP'(-1)\ge0,\\
		\CP'(1)=B+6A>0\ \Rightarrow\ \CP(1)=3A+B+C>0\ \stackrel{\eqref{defsigma}}{\Rightarrow}\ \sigma_y(x,1)\equiv0,\\
		\CP'(-1)=B-6A<0\ \Rightarrow\ \CP(-1)=3A-B+C>0\ \stackrel{\eqref{defsigma}}{\Rightarrow}\ \sigma_y(x,-1)\equiv0.
	\end{array}\right.
	$$
	In any of these three (exhaustive) cases, it happens that
	\begin{equation}\label{A0}
		\CP'(1)\sigma_y(x,1)^2-\CP'(-1)\sigma_y(x,-1)^2\le0\qquad\forall x\in\R.
	\end{equation}
	
	From \eqref{defsigma} and by differentiating some products, we also infer that
	$$
	\begin{array}{cc}
		\psi_x=\CP(y)\sigma_x\, ,\quad\psi_{xx}=\CP(y)\sigma_{xx}\, ,\quad\psi_{xyy}=\CP(y)\sigma_{xyy}+2\CP'(y)\sigma_{xy}+6A\sigma_x\, ,\\
		\psi_{xxx}=\CP(y)\sigma_{xxx}\, ,\quad\psi_{yyy}=\CP(y)\sigma_{yyy}+3\CP'(y)\sigma_{yy}+18A\sigma_y\, ,\quad\psi_{xxxx}=\CP(y)\sigma_{xxxx}\, ,
	\end{array}
	$$
	so that \eqref{bihaint-f} becomes
	\begin{eqnarray}
		&&\CP(y)\sigma_{yyy}(x,y)+3\CP'(y)\sigma_{yy}(x,y)+18A\sigma_y(x,y)+2\big[\CP(y)\sigma_{xx}(x,y)\big]_y+\int_{-1}^{y}\CP(t)\sigma_{xxxx}(x,t)dt \notag\\
		&&=\int_{-1}^{y}\CP(t)\big[\CP(t)\sigma_{xyy}(x,t)+2\CP'(t)\sigma_{xy}(x,t)+\CP(t)\sigma_{xxx}(x,t)\big]dt+\psi_{yyy}(x,-1) \notag\\
		&&\quad + f(x,-1) - f(x,y) + \int_{-1}^y g_x(x,t) dt \notag \\
		&&=\CP(y)^2\sigma_{xy}(x,y)+\int_{-1}^{y}\CP(t)^2\sigma_{xxx}(x,t)dt+\underbrace{\psi_{yyy}(x,-1)-\CP(-1)^2\sigma_{xy}(x,-1) + f(x,-1) }_{H(x)} \notag \\
			&&\quad - f(x,y) + \int_{-1}^y g_x(x,t) dt
		\label{milan}
	\end{eqnarray}
	for every $x\in\R$.
	
	\medskip
	
	\underline{Step 2: Multiplication by $\sigma_y$ and integration with respect to $y$.}
	\smallskip
	
	We preliminarily notice that for the last term in \eqref{milan}$_3$, \eqref{sigmapm1} gives
	$$
	\int_{-1}^{1}H(x)\sigma_y(x,y)dy=H(x)\big[\sigma(x,y)\big]_{-1}^1=0.
	$$
	Hence, if we multiply both sides of \eqref{milan} by $\sigma_y(x,y)$ and we integrate (by parts) for $y\in(-1,1)$, we obtain
	\begin{align}
		&\quad \int_{-1}^{1}\CP(y)\sigma_{yyy}(x,y)\sigma_y(x,y)dy+3\int_{-1}^{1}\CP'(y)\sigma_{yy}(x,y)\sigma_y(x,y)dy+18A\int_{-1}^1\sigma_y(x,y)^2dy \notag\\
		& \quad -2\int_{-1}^{1}\CP(y)\sigma_{xx}(x,y)\sigma_{yy}(x,y)dy-\int_{-1}^{1}\CP(y)\sigma_{xxxx}(x,y)\sigma(x,y)dy \notag \\
		&=\frac{1}{2}\int_{-1}^{1}\CP(y)^2\big[\sigma_{y}(x,y)^2\big]_xdy-\int_{-1}^{1}\CP(y)^2\sigma_{xxx}(x,y)\sigma(x,y)dy \notag \\
		&\quad  - \int_{-1}^1 f(x,y) \sigma_y(x,y) dy - \int_{-1}^1 g_x(x,y) \sigma(x,y) dy \label{milann-f}
	\end{align}
	because of \eqref{sigmapm1}. We proceed through some more integration by parts. We first find
	\begin{eqnarray*}
		\int_{-1}^{1}\CP'(y)\sigma_{yy}(x,y)\sigma_y(x,y)dy &=& \int_{-1}^{1}\frac{\CP'(y)}{2}\Big[\sigma_y(x,y)^2\Big]_ydy\\
		&=& -3A\int_{-1}^{1}\!\!\sigma_y(x,y)^2dy+\frac{\CP'(1)}{2}\sigma_y(x,1)^2-\frac{\CP'(-1)}{2}\sigma_y(x,-1)^2.
	\end{eqnarray*}
	Then, using \eqref{boundary}, we notice that
	$$
	\int_{-1}^{1}\!\!\CP(y)\sigma_{yyy}(x,y)\sigma_y(x,y)dy=-\int_{-1}^{1}\CP(y)\sigma_{yy}(x,y)^2dy-\int_{-1}^{1}\!\!\CP'(y)
	\sigma_{yy}(x,y)\sigma_y(x,y)dy.
	$$
	Therefore, the first line in \eqref{milann-f} becomes
	\begin{eqnarray}
		&&\int_{-1}^{1}\CP(y)\sigma_{yyy}(x,y)\sigma_y(x,y)dy+3\int_{-1}^{1}\CP'(y)\sigma_{yy}(x,y)\sigma_y(x,y)dy+18A\int_{-1}^1\sigma_y(x,y)^2dy \notag\\
		&=&-\int_{-1}^{1}\CP(y)\sigma_{yy}(x,y)^2dy+2\int_{-1}^{1}\CP'(y)\sigma_{yy}(x,y)\sigma_y(x,y)dy+18A\int_{-1}^{1}\sigma_y(x,y)^2dy
		\notag\\
		&=&-\int_{-1}^{1}\CP(y)\sigma_{yy}(x,y)^2dy+12A\int_{-1}^{1}\sigma_y(x,y)^2dy+\CP'(1)\sigma_y(x,1)^2-\CP'(-1)\sigma_y(x,-1)^2 \notag\\
		&\le&-\int_{-1}^{1}\CP(y)\sigma_{yy}(x,y)^2dy +12A\int_{-1}^{1}\sigma_y(x,y)^2dy\label{milanfirst-f}
	\end{eqnarray}
	where, to obtain the inequality in the  last line we used \eqref{A0}.\par
	
	\medskip
	\underline{Step 3: Integration with respect to $x$.}
	\smallskip
	
	We fix $L>0$, we put $Q_L\doteq(-L,L)\times(-1,1)$, and we integrate \eqref{milann-f} for $x\in(-L,L)$ by taking into account \eqref{milanfirst-f}.
	We obtain
	\begin{align}
	&\quad \int_{Q_L}\Big[\CP(y)\sigma_{yy}(x,y)^2+2\CP(y)\sigma_{xx}(x,y)\sigma_{yy}(x,y)+\CP(y)\sigma_{xxxx}(x,y)\sigma(x,y)\Big]dxdy \notag\\
	&\le \int_{Q_L}\CP(y)^2\sigma_{xxx}(x,y)\sigma(x,y)dxdy-\frac{1}{2}\int_{Q_L}\CP(y)^2\big[\sigma_{y}(x,y)^2\big]_xdxdy \notag \\
	&\quad + 12A\int_{Q_L} \sigma_y(x,y)^2 dxdy+ \int_{Q_L} \Big[ f(x,y) \sigma_y(x,y) + g_x(x,y) \sigma(x,y) \Big] dxdy.    \label{pass-f}
\end{align}
	For the second term in the first line, we perform integration by parts several times with the help of \eqref{sigmapm1}, and get
\begin{align} %\label{shanghai-1}
	&\quad \ \int_{Q_L} 2\CP(y)\sigma_{xx}(x,y)\sigma_{yy}(x,y) dxdy \nonumber \\
	&= - \int_{Q_L} \Big[ 2\CP'(y) \sigma_{xx}(x,y) \sigma_y(x,y) + 2 \CP(y) \sigma_{xxy}(x,y) \sigma_y(x,y) \Big] dxdy \nonumber \\
	&=\int_{Q_L} \Big[ 2\CP'(y) \sigma_x(x,y) \sigma_{xy} (x,y) + 2 \CP(y) \sigma_{xy}(x,y)^2 \Big] dxdy \nonumber \\
	&\quad - \int_{-1}^1 \Big[ 2\CP'(y) \sigma_x(x,y) \sigma_y(x,y) + 2 \CP(y) \sigma_{xy}(x,y) \sigma_y(x,y) \Big]_{-L}^L dy \nonumber \\
	&= - 6A \int_{Q_L} \sigma_x(x,y)^2 dxdy + \int_{Q_L} 2 \CP(y) \sigma_{xy}(x,y)^2 dxdy \nonumber \\
	&\quad - \int_{-1}^1 \Big[ 2\CP'(y) \sigma_x(x,y) \sigma_y(x,y) + 2 \CP(y) \sigma_{xy}(x,y) \sigma_y(x,y) \Big]_{-L}^L dy. \nonumber
\end{align}
	Integrating by parts twice in $x$ for the last term in the first line of \eqref{pass-f}, it becomes
$$
\int_{Q_L}\CP(y)\sigma_{xx}(x,y)^2dxdy+
\int_{-1}^1\CP(y)\big[\sigma_{xxx}(x,y)\sigma(x,y)-\sigma_{xx}(x,y)\sigma_x(x,y)\Big]_{-L}^Ldy
$$
Concerning the second line, we see that
\begin{eqnarray*}
	\int_{Q_L}\CP(y)^2\sigma_{xxx}(x,y)\sigma(x,y)dxdy &=& -\frac{1}{2}\int_{Q_L}\CP(y)^2\big[\sigma_x(x,y)^2\big]_xdxdy\\
	&&+\int_{-1}^1\CP(y)^2\Big[\sigma_{xx}(x,y)\sigma(x,y)\Big]_{-L}^Ldy\\
	&=& \int_{-1}^1\CP(y)^2\bigg[\sigma_{xx}(x,y)\sigma(x,y)-\frac{\sigma_x(x,y)^2}{2}\bigg]_{-L}^Ldy,
\end{eqnarray*}
$$
-\frac{1}{2}\int_{Q_L}\CP(y)^2\big[\sigma_{y}(x,y)^2\big]_xdxdy=-\frac{1}{2}\int_{-1}^1\CP(y)^2\Big[\sigma_{y}(x,y)^2\Big]_{-L}^Ldy,
$$
and, hence,
\begin{eqnarray*}
	&& \int_{Q_L}\CP(y)^2\sigma_{xxx}(x,y)\sigma(x,y)dxdy-\frac{1}{2}\int_{Q_L}\CP(y)^2\big[\sigma_{y}(x,y)^2\big]_xdxdy\\
	&=& \int_{-1}^1\CP(y)^2\bigg[\sigma_{xx}(x,y)\sigma(x,y)-\frac{|\nabla\sigma(x,y)|^2}{2}\bigg]_{-L}^Ldy.
\end{eqnarray*}
Note  that
$$|\nabla^2 \sigma|^2(x,y) = \sigma_{xx}(x,y)^2 + 2\sigma_{xy}(x,y)^2 +\sigma_{yy}(x,y)^2.$$
By collecting terms, from \eqref{pass-f} we arrive at the key a priori estimate
     \begin{align}
     	\Gamma(L) &\doteq - 6 A \int_{Q_L} \sigma_x(x,y)^2 dxdy -12 A \int_{Q_L} \sigma_y(x,y)^2 dxdy  + \int_{Q_L} \CP(y) |\nabla^2\sigma|^2(x,y) dxdy \nonumber \\
     	&\le \int_{-1}^1\CP(y)\Big[\sigma_{xx}(x,y)\sigma_x(x,y)-\sigma_{xxx}(x,y)\sigma(x,y) + 2\sigma_{xy}(x,y) \sigma_y(x,y)\Big]_{-L}^Ldy \nonumber \\
     	&\quad + \int_{-1}^{1} 2\CP'(y) \Big[ \sigma_x(x,y) \sigma_y(x,y) \Big]_{-L}^L dy +\int_{-1}^1\CP(y)^2\bigg[\sigma_{xx}(x,y)\sigma(x,y)-\frac{|\nabla\sigma(x,y)|^2}{2}\bigg]_{-L}^Ldy \nonumber \\
     	&\quad + \int_{-1}^1 \Big[g(x,y) \sigma(x,y)\Big]_{-L}^L dy + \int_{Q_L} \Big[ f(x,y) \sigma_y(x,y) - g(x,y) \sigma_x(x,y) \Big] dxdy \nonumber \\
     	&=: \gamma(L) + F(L), \label{key-apriori-f}
     \end{align}
     where $\gamma(L)$ stands for the sum of the line integrals and $F(L)$ stands for the area integral involving $f$ and $g$.
     Similar to Lemma \ref{lem:poincare} (see \eqref{911a} and \eqref{911b}), we know that
     \begin{equation} \label{GammaLcontrol}
     	\Gamma(L) \gtrsim \int_{Q_L} \big[\sigma_y(x,y)^2 + \sigma_x(x,y)^2 + \psi_{xx}(x,y)^2 + 2\psi_{xy}(x,y)^2 + \psi_{yy}(x,y)^2\big] dxdy.
     \end{equation}
     More generally, for any $b>a>0$ we have
          \begin{equation} \label{GammaLcontrolab}
     	\Gamma(b) - \Gamma(a) \gtrsim \int_{Q_b\setminus Q_a} \big[\sigma_y(x,y)^2 + \sigma_x(x,y)^2 + \psi_{xx}(x,y)^2 + 2\psi_{xy}(x,y)^2 + \psi_{yy}(x,y)^2\big] dxdy.
     \end{equation}
      \par

     \medskip
     	\underline{Step 4: Estimating $\Gamma$.}
     	\smallskip
     	
	Since the problem we are studying is linear, without loss of generality, we can assume that \begin{equation} \label{fX0}
		\|\f\|_{\X^0} \le 1.
		\end{equation}
    Our goal is to prove that there exists a (possibly large)  constant $N \ge 1$ depending only  on $A,B,C$ such that at least one of the following statements is true:
    \begin{enumerate}[(a)]
    	\item  $\Gamma(10) \le N^2$.
    \item For any $\mathbb{Z} \ni L \ge N$, $	\Gamma(L) \ge NL$.
    \end{enumerate}
	
	Assume that (a) fails, namely, $\Gamma(10) > N^2$, where $N$  will be taken sufficiently large according to the arguments below. We shall prove (b) by induction. Clearly, $\Gamma(N) \ge N^2$ is satisfied by our assumption as long as $N \ge 10$. Now, suppose that $\Gamma(L) \ge NL$ is valid for some $L \ge N$, and we would like to give the corresponding lower bound on $\Gamma(L+1)$.
	
	By the monotonicity of $\Gamma$ and \eqref{key-apriori-f}, for any $l \in (L, L+1)$ there holds
	\begin{equation} \label{chain}
		 NL \le \Gamma(L) \le \Gamma(l) \le \gamma(l)+F(l).
	\end{equation}
	By H\"older's inequality, \eqref{GammaLcontrol} and \eqref{fX0}, we have
	\begin{equation}
		F(l) \le \left(\int_{Q_l} \big[f(x,y)^2 + g(x,y)^2\big] dxdy\right)^\frac12 \left(\int_{Q_l} \big[\sigma_x(x,y)^2 + \sigma_y(x,y)^2\big] dxdy\right)^\frac12
		\lesssim L^\frac12 \Gamma(l)^\frac12.
	\end{equation}
	Hence, from the last inequality in \eqref{chain}, we get
	\begin{equation} \label{Gammalgtr}
		\Gamma(l) \lesssim \gamma(l) + L.
	\end{equation}
    By taking $N$ large,  \eqref{Gammalgtr} together with the first two inequalities in \eqref{chain}  implies that
	\begin{equation} \label{gammalge}
		\gamma(l) \gtrsim NL.
	\end{equation}
	In the sequel, $c$ will stand for positive constants depending  on $A,B,C$. Integration of \eqref{gammalge} in $l \in (l_1, l_2)$ with $l_1 \in (L, L+\frac13), \, l_2 \in (L+\frac23, L+1)$, together with the application of H\"older's inequality and \eqref{GammaLcontrolab}, gives
	\begin{align} \label{intge}
	 c\big[\Gamma(L+1) - \Gamma(L)\big] + \int_{Q_{l_2} \setminus Q_{l_1}}  \Big[ g(x,y)^2 - \text{sign}(x) \CP(y) \sigma_{xxx}(x,y) \sigma(x,y)\Big] dx dy \gtrsim NL.
	\end{align}
	By \eqref{fX0}, the third term in \eqref{intge} satisfies
	\begin{equation} \label{gsmall}
		\int_{Q_{l_2}\setminus Q_{l_1}} g(x,y)^2 dxdy \le 1 \ll NL.
	\end{equation}
	Performing integration by parts, the fourth term  in \eqref{intge} satisfies
	\begin{align} \label{furtherbdry}
		&\quad - \int_{Q_{l_2}\setminus Q_{l_1}}  \text{sign}(x) \CP(y) \sigma_{xxx}(x,y) \sigma(x,y) dxdy \nonumber \\
		&=  \int_{Q_{l_2}\setminus Q_{l_1}}  \text{sign}(x) \CP(y) \sigma_{xx}(x,y) \sigma_x(x,y) dxdy - \int_{-1}^1 \Big[  \CP(y) \sigma_{xx}(x,y) \sigma(x,y) \Big]_{l_1}^{l_2} dy \nonumber \\
		&\quad + \int_{-1}^1 \Big[  \CP(y) \sigma_{xx}(x,y) \sigma(x,y) \Big]_{-l_2}^{-l_1} dy \nonumber \\
		&\lesssim \Gamma(L+1) - \Gamma(L) + \sum_{x  \in \{-l_2, -l_1, l_1, l_2\}} \int_{-1}^1 |\CP(y) \sigma_{xx}(x,y) \sigma(x,y) | dy\,.
	\end{align}
	Combining \eqref{intge}--\eqref{furtherbdry}, we arrive at
	\begin{equation} \label{l1l2}
		c \big[\Gamma(L+1) - \Gamma(L)\big] + \sum_{x  \in \{-l_2, -l_1, l_1, l_2\}} \int_{-1}^1 |\CP(y) \sigma_{xx}(x,y) \sigma(x,y) | dy \gtrsim NL.
	\end{equation}
		Now, integrating \eqref{l1l2}  in $l_1 \in (L, L+\frac13), \, l_2 \in (L+\frac23, L+1)$ and using
	\begin{equation}
		\int_{Q_{L+1} \setminus Q_L} |\CP(y) \sigma_{xx}(x,y) \sigma(x,y) | dxdy \lesssim \Gamma(L+1) - \Gamma(L),
	\end{equation}
	we obtain
	\begin{equation}
		c \big[\Gamma(L+1) - \Gamma(L)\big] \gtrsim NL.
	\end{equation}
	Since $L \ge N$, we have $NL \gg N$, which then leads to the desired estimate in (b), $\Gamma(L+1) \ge N(L+1)$.
	
	\medskip
	     	\underline{Step 5: Existence and uniqueness.}

	     	\smallskip
	     	Now, we consider any $\f \in \X^0$ with $\|\f\|_{X^{0}} = 1$, and for $M \ge 1$ define the truncated forces
	     	\begin{equation}
	     		\f^{(M)}(x,y) := \begin{cases}
	     			\f(x,y), \quad \text{if} \ \ |x|<M, \\
	     			0, \quad \text{otherwise}.
	     		\end{cases}
	     	\end{equation}
	     	Then clearly $\f^{(M)} \in L^2(S)$ and
	     	\begin{equation}
	     		\|\f^{(M)}\|_{X^{0}} \le 1.
	     	\end{equation}
	     	By Theorem \ref{L2}, there exists a unique solution $\Q^{(M)} \in H^2(S)$ to \eqref{linearproblem}. Denote the corresponding energy functional by $\Gamma^{(M)}(L)$. By the density of smooth functions in $L^2(S)$, the estimates on $\Gamma$ in Step 4 also apply to $\Gamma^{(M)}$. Since for each $M$, $\Gamma^{(M)}(L)$ is uniformly bounded in $L$  by Theorem \ref{L2}, we know that the alternative $(a)$ has to hold, that is, $\Gamma^{(M)}(10) \le N^2$, where the constant~$N\ge1$ depends on $A,B,C$ only. By
		\eqref{GammaLcontrol} and local Stokes estimates (see \eqref{stokes1}), we have
	     	\begin{equation}
	     		\|\Q^{(M)}\|_{H^2\bigl((-9,9)\times (-1,1)\bigr)} \le c  \|\f\|_{\X^0}.
	     	\end{equation}
	     	 By translation invariance of the problem \eqref{linearproblem} and the $\X$ norms, we then obtain the a priori bound
	     		\begin{equation} \label{X20N}
	     		\|\Q^{(M)}\|_{\X^2} \le c \|\f\|_{\X^0}.
	     	\end{equation}
	     	Taking $M \to +\infty$,  $\Q^{(M)}$ converges (in the sense of distributions) along some subsequence to a limiting vector field $\Q$ satisfying \eqref{X20}. Clearly, $\Q$ solves the linear problem \eqref{linearproblem}.
	     	
	     	It remains to show uniqueness. Suppose that $\f \equiv 0$ and $\Q \in \X^2$ solves \eqref{linearproblem}. By local Stokes estimates, we know that
	        	\begin{equation}\label{uniq-1}
	        	\sup_{L \in \R} \gamma(L) <+\infty.
	        \end{equation}
	     	which implies, via \eqref{key-apriori-f}, that
	     		\begin{equation}\label{uniq-2}
	     		\sup_{L \in \R} \Gamma(L) <+\infty.
	     	\end{equation}
	     	By \eqref{GammaLcontrol} and local Stokes estimates, this leads to $\Q \in H^2(S)$. Hence, by the uniqueness statement in Theorem \ref{L2}, we get $\Q \equiv 0$.	     	
\end{proof}
\begin{remark}
	{\sl As noticed previously, the class of solutions that are periodic in the $x$-direction with a velocity field twice differentiable in $y$, constitutes a subclass of $\mathbb H(S)$. It is therefore important to emphasize that, unlike the main results demonstrated in \cite{SWX1,SWX2} in such a subclass, Theorem \ref{bounded} does not require {\em any restrictions} on the period, the size of $A, B, C$, or on $\Phi$.}
\end{remark}

Combining Theorem \ref{bounded} with the Contraction-Mapping Theorem, and
observing that, by embedding,
$$
\|\Q \cdot \nabla \Q\|_{\X^0} \lesssim \|\Q\|_{\X^2}^2\,,
$$
we may proceed exactly as in the proof of Theorem \ref{cor1} to show the following result.
\begin{theorem} \label{cor2} Assume \eqref{ABC}. Then, there is
$\eta>0$ depending on $A,B,C$ such that for any $\f \in \X^0$ with $\|\f\|_{\X^0} < \eta$, there is a unique corresponding solution $(\U,q) \in \mathbb H(S)$ to the nonlinear problem \eqref{nonlinearproblem} of the form \eqref{form} satisfying
$$
\| \Q\|_{\X^2}+\|\nabla q\|_{\X^0} \le  \gamma_1\,\|\f\|_{\X^0}\,,
$$
with $\gamma_1>0$ depending on $A,B$ and $C$.
\end{theorem}
We finally notice that  if $(\Q,q)\in \mathbb H(S)$ and $\|\Q\|_{\X^1}$ is below a certain constant, from the Stokes estimates \eqref{StEs}, it follows that
$$
\|\Q\|_{\X^2}\lesssim \|\Q\|_{\X^1}\,.
$$
Thus, arguing exactly as in the proof of Corollary \ref{Cor1} we may show the following one.

\begin{corollary}\label{coroll}
Let the assumption of Theorem \ref{cor2} hold. Then, there exists $\delta_0>0$, depending on $A,B,C$, such that the only solution to \eqref{nonlinearproblem} with $\mathbf{f}\equiv\0$ of the type
$$
\mathbf{u}=\mathbf{u}_*+\Q\,,\ \ p=p_*+q
$$
where $(\Q,q)\in \mathbb H(S)$ with $\|\Q\|_{\X^1}\le\delta_0$, is the Couette--Poiseuille flow $\mathbf{u}=\mathbf{u}_*\,, p=p_*$\,.
\label{Cor2}
\end{corollary}
\begin{remark} {\sl As observed previously, the Couette-Poiseuille flow belongs to the class $\mathbb H(S)$. Therefore, we can reformulate Corollary~\ref{Cor2} by stating that, under its assumptions, there is no solution ``close" to the Couette-Poiseuille solution. }
\end{remark}

\section{Symmetric flows}\label{sec:6}

In 2D symmetric domains, a solenoidal vector field $\U=(u_1,u_2)$ may be symmetric, only if its components satisfy some ``combined
symmetries'' that allow to maintain the solenoidal condition. There are four possible cases:
\begin{equation}
\begin{array}{c}
(X1)\ u_1\mbox{ is }x\mbox{-even and }u_2\mbox{ is }x\mbox{-odd}\qquad(X2)\ u_1\mbox{ is }x\mbox{-odd and }u_2\mbox{ is }x\mbox{-even}\\
(Y1)\ u_1\mbox{ is }y\mbox{-even and }u_2\mbox{ is }y\mbox{-odd}\qquad(Y2)\ u_1\mbox{ is }y\mbox{-odd and }u_2\mbox{ is }y\mbox{-even.}
\end{array}\label{Xy}
\end{equation}
Clearly, for the particular problem \eqref{poiseuille2d}, which is $x$-translation-invariant, the $x$-symmetry can be replaced with the symmetry
with respect to the line $x=c$ for some $c\in\R$.\par
By suppressing the odd term in \eqref{parabolasign}, namely, specializing  \eqref{ABC} to the following
\begin{equation}\label{AC}
B=0,\qquad (A,C)\neq(0,0),\qquad A\le0,\qquad C\ge3|A|,
\end{equation}
we obtain the symmetric (even) flow
\begin{equation}\label{evenflow}
\CP(y)=3Ay^2+C
\end{equation}
that allows for tangential flows on $\partial S$, provided they have the same sign.
As a ``limit'' case, \eqref{evenflow} includes the Poiseuille flow $\CP(y)=-3A(1-y^2)$ and the constant flow $\CP(y) \equiv C$.\par
The components $u_1=u_1(x,y)$ and $u_2=u_2(x,y)$ of flow such as \eqref{evenflow} enjoy two of the above symmetry properties:\par
$$(u_1,u_2)\mbox{ satisfies }(X1),\ (Y1).$$
These symmetry assumptions allow us to simplify essentially the proofs of Theorems \ref{L2}-\ref{bounded}, see Remark \ref{moresymm}.
What is much more important, under these symmetry assumptions a stronger uniqueness result in the large can be proved.

\subsection{Uniqueness of symmetric  solutions to the nonlinear problem}\label{sec:6.1}

So far, we have proved nonexistence of solutions $(\U,p)$ to \eqref{poiseuille2d}$_{1,2}$ in the form \eqref{form} for $\Q$ sufficiently small. Of course, this does not  exclude that solutions to \eqref{nonlinear}, other than \eqref{Poiseuille} and belonging to some ``far"
branch of solutions, do exist. We do not have a full answer to this question. However, when the Couette-Poiseuille flow is itself symmetric (even,
satisfying \eqref{AC}) something can be said. In fact, under this symmetry assumption, we prove that the only symmetric solution to
\begin{equation}\label{genpoiseuille}
\left\{\begin{array}{ll}\smallskip
-\Delta\U+ (\U\cdot\nabla)\U+\nabla p=\0\quad\mbox{in }S,\\[6pt]
\nabla\cdot\U=0\quad\mbox{in }S,\\[6pt]
\U(x,\pm1)=(3A+C)\mathbf{e}_1\quad \forall x\in\R,\\[6pt]
\displaystyle
\int_{-1}^1\U(0,y)\cdot\mathbf{e}_1\,dy=2(A+C)\, ,
\end{array}\right.
\end{equation}
is the Couette-Poiseuille flow:
$$
\U_*=\mathcal F(y)\,\mathbf{e}_1\,,\ \ p_*=6A\,x\,,
$$
with $\mathcal F$ given in \eqref{evenflow}.
Precisely, we have the following result.
\begin{theorem}\label{nosymm}
Assume \eqref{AC}. Then, the only solution to \eqref{genpoiseuille}  of the type
$$
\mathbf{u}=\mathbf{u}_*+\Q\,,\ \ p=p_*+q
$$
where $(\Q,q)\in \mathbb H(S)$
and
$$
\mbox{either $\Q$ satisfies $(Y2)$ or $\Q$ satisfies $(X1)$\,,}
$$
is the Couette--Poiseuille flow $\mathbf{u}=\mathbf{u}_*\,, p=p_*$
\end{theorem}
\begin{proof} It follows the line of the proof of Theorem \ref{bounded} with the addition of the nonlinear term that, as we shall see,
cancels out if $\Q$ satisfies one of above symmetry assumptions. \par
Consider again the stream function $\psi$ so that \eqref{nopressure} becomes
$$
\Delta^2\psi- \big((3Ay^2+C)(\psi_{xyy}+\psi_{xxx})-6A\psi_x\big)+\big(\psi_x\Delta\psi_y-\psi_y\Delta\psi_x\big)=0\quad\mbox{in }S\, .
$$
By integrating this equation over $(-1,y)$, we obtain \eqref{bihaint-f} with an additional term, that is,
\begin{eqnarray}
&&\psi_{yyy}(x,y)+2\psi_{xxy}(x,y)+\int_{-1}^{y}\psi_{xxxx}(x,t)dt \notag\\
&&-\int_{-1}^{y}\big((3At^2+C)(\psi_{xyy}(x,t)+\psi_{xxx}(x,t))-6A\psi_x(x,t)\big)dt\notag\\
&&+\int_{-1}^{y}\big(\psi_x(x,t)\Delta\psi_y(x,t)-\psi_y(x,t)\Delta\psi_x(x,t)\big)dt=c(x)\label{bihaint2}
\end{eqnarray}
Let $\psi(x,y)=(3Ay^2+C)\sigma(x,y)$, then \eqref{bihaint2} becomes \eqref{milan} with the additional terms
\begin{eqnarray*}
&&\int_{-1}^{y}(3At^2+C)\sigma_x(x,t)\bigg[(3At^2+C)\sigma_{xxy}(x,t)+6At\sigma_{xx}(x,t)\\
&& \hspace{1cm} +(3At^2+C)\sigma_{yyy}(x,t)
+18At\sigma_{yy}(x,t)+18A\sigma_x(x,t)\bigg]dt\\[4pt]
&&-\int_{-1}^{y}\big[(3At^2+C)\sigma_y(x,t)+6At\sigma(x,t)\big]\bigg[(3At^2+C)\sigma_{xxx}(x,t)\\
&& \hspace{1cm} +(3At^2+C)\sigma_{xyy}(x,t)+12At\sigma_{xy}(x,t)
+6A\sigma_x(x,t)\bigg]dt
\end{eqnarray*}
which, multiplied by $\sigma_y(x,y)$ and integrated by parts for $y\in(-1,1)$ (recall \eqref{sigmapm1}), gives
\begin{eqnarray*}
&&-\int_{-1}^{1}(3Ay^2\!+\!C)\sigma_x(x,y)\bigg[(3Ay^2+C)\sigma_{xxy}(x,y)+6Ay\sigma_{xx}(x,y)\\
&& \hspace{1cm} +(3Ay^2+C)\sigma_{yyy}(x,y)+18Ay\sigma_{yy}(x,y)+18A\sigma_x(x,y)\bigg]\sigma(x,y)dy\\
&&+\int_{-1}^{1}\big[(3Ay^2+C)\sigma_y(x,y)+6Ay\sigma(x,y)\big]\bigg[(3Ay^2+C)\sigma_{xxx}(x,y) \\
&& \hspace{1cm} +(3Ay^2+C)\sigma_{xyy}(x,y)+12Ay\sigma_{xy}(x,y)
+6A\sigma_x(x,y)\bigg]\sigma(x,y)dy.
\end{eqnarray*}

Thus, writing $\mathbf{u}=\mathbf{u}_*+\Q$, the following holds.
If $\Q$ satisfies $(Y2)$, then the stream function $\psi$ is $y$-even and
also $\sigma$ is $y$-even. This implies that all the integrands in the above integrals are $y$-odd. Therefore, the integrals
are null and we obtain again \eqref{milanfirst-f} so that we may conclude
as in the proof of uniqueness in Theorem~\ref{bounded}, by integrating with respect to~$x$ and moving on to estimates~(\ref{uniq-1})--(\ref{uniq-2}).\par
If $\Q$ satisfies $(X1)$, then the stream function $\psi$ is $x$-even and also $\sigma$ is $x$-even. Hence, the above $x$-dependent integrals
are $x$-odd and, after integrating for $x\in(-L,L)$, the resulting integrals are null. Hence, we obtain again \eqref{pass-f}--\eqref{key-apriori-f}, and then conclude as in Theorem \ref{bounded}.\end{proof}
\par
We now observe that the Poiseuille flow \eqref{Poiseuille}--\eqref{PF}$_1$ satisfies \eqref{ABC}, belongs to the space $\mathbb H(S)$, and also falls into symmetry class $(X1)$. Therefore, Theorem \ref{nosymm} guarantees, as a particular case, the following result.
\begin{corollary}\label{Pos} {Let $\Phi$ be an arbitrary positive number.\footnote{See footnote \ref{foot:p}.} Then, the problem
$$
\left\{\begin{array}{ll}\smallskip
-\Delta\U+ (\U\cdot\nabla)\U+\nabla p=\0 \quad\mbox{in }S,\\[6pt]
\nabla\cdot\U=0 \quad\mbox{in }S,\\[6pt]
\U(x,\pm1)=\0\quad \forall x\in\R,\\[6pt]
\displaystyle \int_{-1}^1\U(0,y)\cdot\mathbf{e}_1\,dy=\Phi,
\end{array}\right.
$$
has one and only one solution $(\U,p)\in \mathbb H(S)\cap (X1)$, given by the Poiseuille flow
$$
\mathbf{u}_\Phi(x,y)=\mbox{$\frac34$}\Phi\,(1-y^2)\mathbf{e}_1\,,\ \ p_\Phi(x,y)=-\mbox{$\frac14$}\Phi\,x\,.
$$}
\end{corollary}

We conclude this section by bringing further evidence to the relevance of symmetry.

\begin{remark}\label{moresymm}
Under the assumptions of Corollary \ref{coroll} plus the symmetry assumption that $\Q$ satisfies either $(X1)$ or $(X2)$, the proof considerably simplifies.
Indeed, we associate again to $\Q=(v,w)$ a stream function $\psi\in H^2_0(S)$ and we decompose the space $H^2_0(S)$ as direct sum of the two orthogonal
closed subspaces of $x$-even and $x$-odd functions, that is, $H^2_0(S)=\HH_e\oplus\HH_o$.
Next, we observe that both (X1) and (X2) imply that $\psi_x\psi_y$ is $x$-odd. In particular, this implies that
\begin{equation}\label{technical}
\int_S\CP(y)\psi_{yy}\psi_x =\int_S\big(\CP(y)\psi_{xx}-6A\psi\big)\psi_x=0\qquad\forall\psi\in \HH_e \cup \HH_o
\end{equation}
and what Rabier \cite[p.353]{rabier2} calls ``the only obstacle to get a convenient estimate'' is removed. We also introduce the spaces of distributions
$$
H^{-2}_e:=\{T\in H^{-2}(S); \mbox{ ker}T\supseteq\HH_o\},\quad H^{-2}_o:=\{T\in H^{-2}(S); \mbox{ ker}T\supseteq\HH_e\}
$$
and we weakly reformulate \eqref{biha-f} (with $f=g=0$) in these subspaces. We first prove that
\begin{equation}\label{weakeven}
\forall T_e\in H^{-2}_e\quad\exists!\psi\in\HH_e\mbox{ s.t. }
\int_S\Delta\psi\Delta\phi+K\int_S\big(\CP(y)(\psi_{yy}+\psi_{xx})-6A\psi\big)\phi_x=\langle T_e,\phi\rangle\quad\forall
\phi\in\HH_e\, .
\end{equation}
By taking $\phi=\psi$ in \eqref{weakeven} and using \eqref{technical}, we obtain
\begin{equation}\label{selftest}
\int_S|\Delta\psi|^2=\langle T_e,\psi\rangle\le\|T_e\|_{H^{-2}_e}\|\psi\|_{\HH_e}\qquad\forall \psi\in\HH_e.
\end{equation}
It is known \cite[Theorem 2.2]{gazgruswe} that $\|\Delta\cdot\|_{L^2(\Omega)}$ is a norm equivalent to the $H^2(\Omega)$-norm
in $H^2_0(\Omega)$ when $\Omega\subset\R^2$ is a domain where the Poincaré inequality holds as for $S$.
Then \eqref{selftest} is an {\em a priori bound} for solutions to \eqref{weakeven} and the Lax-Milgram Theorem for linear biharmonic problems with lower order perturbations
\cite[Theorem 2.15]{gazgruswe} applies. Existence and uniqueness for \eqref{weakeven} follow, showing the bijectivity of the Fréchet derivative.
Similarly, but with obvious changes, one can prove \eqref{weakeven} also in $H^{-2}_o$.\end{remark}

\subsection{Non-invertibility of the linearization around some symmetric flows}\label{sec:6.2}

The proof of Theorem \ref{bounded} strongly uses \eqref{ABC} which implies the positivity of the function $\CP=\CP(y)$ in $(-1,1)$, see
\eqref{parabolasign}, that is, absence of flow reversal. It is then natural to wonder if a similar invertibility result also holds if \eqref{ABC} fails. We exhibit here
a counter-example showing that this is not the case for some reverse flow.

\begin{theorem} \label{prop:general}
There exist $A<0$, $C<3|A|$, $T >0$ such that the linearized operator at the Couette-Poiseuille flow
\begin{equation} \label{poi-drift}
\CP(y)=3A y^2 + C \, ,
\end{equation}
in the periodic channel $x \in [0, 2 \pi/T]/\sim$, $y \in [-1,1]$ with the corresponding boundary conditions, is not injective.
\end{theorem}
\begin{proof} Related to \eqref{OS}, consider the following eigenvalue problem for the Orr-Sommerfeld operator
\begin{equation}\label{OS-2}
\left\{\begin{array}{ll}
\!\!\!\phi''''(y)-2 T^2\phi''(y) + T^4\phi(y) + 3 A T i \, \Big[(1-y^2)(\phi''(y)- T^2\phi(y))+2\phi(y)\Big]=\lambda (\phi''(y)-T^2\phi(y)) \\
\!\!\!\phi(\pm1)=\phi'(\pm1)=0\,,
\end{array}\right.
\end{equation}
where $A<0$ and $T >0$ are constants to be determined.\par
By the result of Grenier, Guo and Nguyen \cite{GGN} on the spectral instability of Poiseuille flows,
\begin{equation}\label{instab}
\begin{array}{cc}
\exists\ A_0 <0\ T_0>0\mbox{ s.t. \eqref{OS-2} with $A = A_0,\ T = T_0$ admits a nonzero eigenfunction}\\
\mbox{whose corresponding eigenvalue has strictly positive real part.}
\end{array}
\end{equation}
On the other hand, we claim that
\begin{equation}\label{claim}
\mbox{if $|AT|$ is sufficiently small, then \eqref{OS-2} admits only eigenvalues $\lambda$ with $\text{Re} \, \lambda < 0$.}
\end{equation}
To see this, we test \eqref{OS-2}$_1$ with $\overline{\varphi}$ and integrate by parts in $y$ to obtain
\begin{align} \label{OS-2-energy}
\int_{-1}^1 (|\varphi''|^2 + (2 T^2 + \lambda)  |\varphi'|^2 + (T^4 + \lambda T^2) |\varphi|^2) dy  &=  3 A T i  \int_{-1}^1 \left[(1-y^2) (|\varphi'|^2 + T^2 |\varphi|^2) - 2 |\varphi|^2\right] dy  \nonumber \\
&\quad - 6 A T i \int_{-1}^1 y \varphi' \overline{\varphi} dy.
\end{align}
Suppose that $\text{Re} \,\lambda \ge 0$, then taking the real part of  \eqref{OS-2-energy} we get
\begin{align} \label{OS-2-energy-2}
\int_{-1}^1 |\varphi''|^2 dy \le \int_{-1}^1 \left[|\varphi''|^2 + (2 T^2 + \text{Re} \,\lambda)  |\varphi'|^2 + (T^4 + T^2 \text{Re} \,\lambda ) |\varphi|^2\right] dy  &=   6 A T \, \text{Im}  \int_{-1}^1 y \varphi' \overline{\varphi} dy.
\end{align}
Using the H\"older and Poincar\'e-Wirtinger inequalities, we find that
$$
\left|\int_{-1}^1 y \varphi' \overline{\varphi} dy\right| \le \left(\int_{-1}^1 |\varphi'|^2 dy\right)^\frac12 \left(\int_{-1}^1 |\varphi|^2 dy\right)^\frac12 \le c \int_{-1}^1 |\varphi''|^2 dy
$$
which, combined with \eqref{OS-2-energy-2}, shows that if $|AT|$ is sufficiently small, then $\varphi \equiv 0$; so \eqref{claim} is proved.\par
	
Note that the discrete spectrum of the operator (unbounded and closed on $H_0^1$, with domain $H^2_0\cap H^3$)
\begin{equation}
(\partial_y^2 - T^2)^{-1} \left\{ (\partial_y^2 -T^2)^2 + 3AT i \left[(1-y^2) (\partial_y^2 - T^2) + 2\right]  \right\}
\end{equation}
depends continuously on the parameter $A$ (see, e.g., \cite{Kato}), where $(\partial_y^2 -T^2)^{-1}$ is defined using Dirichlet boundary conditions. Hence, by \eqref{instab}-\eqref{claim}, we know that there exists some $A_0 < A_1 < 0$ such that, for $A = A_1, T = T_0$, the problem \eqref{OS-2} admits an eigenvalue $\lambda_1$ with $\text{Re}\, \lambda_1 = 0$. In this case, \eqref{OS-2} can be written as
$$
\left\{\begin{array}{ll}
\!\!\!\phi''''(y)-2 T_0^2\phi''(y) + T_0^4\phi(y) - 3 A_1 T_0 i \, \Big[(-1 + \frac{\text{Im} \lambda_1}{3 A_1 T_0}+y^2)(\phi''(y)- T_0^2\phi(y)) - 2\phi(y)\Big]=0 \\
\!\!\!\phi(\pm1)=\phi'(\pm1)=0\,.
\end{array}\right.
$$
Hence, we obtain the conclusion with $A = A_1$, $C = -3A_1+ \frac{\text{Im} \lambda_1}{T_0}$, $T = T_0$.
From Theorem \ref{bounded} we infer that, necessarily, $\text{Im} \lambda_1<0$ (otherwise the condition \eqref{ABC} is satisfied so that the linearization at the flow \eqref{poi-drift} is invertible), which completes the proof.
\end{proof}

\begin{remark}
	{\sl It is still unclear how to prove the existence of stationary solutions bifurcating  from the flows \eqref{poi-drift}, which are equivalent to traveling wave solutions bifurcating from the standard Poiseuille flow. Due to the translation invariance of the problem along $x$, the null space of the linearized operator at \eqref{poi-drift} is always even dimensional (as a real vector space). Hence, the classical Krasnoselskii--Rabinowitz Theorem \cite{Kras,Rabinowitz} (see also \cite[Chapter 22]{Brown}) does not apply in this situation.}
\end{remark}

\section{Some historical facts} \label{hfacts}
In March 1998, G.P. Galdi gave a seminar at the mathematics department of the University of Pittsburgh on the possible existence of bounded solutions, different from the classical Poiseuille solution (that is, $\CP(y) = -3A(1 - y^2)$), in an infinite two-dimensional straight channel. In particular, he formulated the boundary value problem (\ref{OS}) with $\CP$ as above and $h_\xi(y)\equiv0$, and raised the question of whether it could admit a non-trivial solution. The seminar was attended by W.C. Troy, P. Rabier and J.B. McLeod, among others.  A few weeks later, W.C. Troy presented some numerical results that suggested that the solution should be identically equal to zero. Subsequently, in October 1998, J.B. McLeod wrote a private letter to G.P. Galdi attempting to provide a rigorous proof of this result. However, his proof contained a flaw, which was subsequently and successfully corrected in a new proof provided by J.B. McLeod in a letter dated November 1998 addressed to G.P. Galdi \cite{mcleod}. Apparently, such an important result remained unpublished and, certainly, was not adequately disseminated.
In this regard, in 2002, P. Rabier
in the note added in proof  \cite[p.373]{rabier2} writes that ``J.B. McLeod was recently able to show in
{\em Eigenvalue stability of stationary Poiseuille flow (preprint)} that these operators are one-to-one for any flow. It thus follows from our Lemma 2.2 (i) that they are isomorphisms, which, as also shown in this paper, implies the invertibility
of the Poiseuille linearization for all values of $\R$." In fact, there is no Lemma 2.2 (i) in \cite{rabier2} and Rabier probably refers
to Lemma 3 (ii), a very powerful result indeed! Moreover, there is no trace on the web of a preprint (nor a published version!) of such notes
by McLeod. Many years later, Rabier \cite{rabier3} clarified that he ``only heard the result by McLeod in a seminar given by him''.
Apparently, this paper was never published and, according to \cite{Has}, these notes are contained in \cite[Sections 10.4 and 10.5]{HMC}. Quite surprinsingly, at present date, the book \cite{HMC} has 37 citations on MathSciNet and 108 citations on Google Scholar, but none of them within the Mathematical Fluid Mechanics community. This seems to demonstrate that the proof contained in \cite{mcleod} is completely unknown to the mathematicians in question, and we hope that this article will help to disseminate it.

\par\bigskip\noindent
{\bf Acknowledgements.} The research of Giovanni P. Galdi is supported by US National Science Foundation, Grant DMS 2307811. The research of Filippo Gazzola is supported by the grant \textit{Dipartimento di Eccellenza 2023-2027},
issued by the Ministry of University and Research (Italy); he is also partially supported by INdAM. The research of Xiao Ren is partially supported by the \textit{National Natural Science Foundation of China} (No. 62588101) and the \textit{National Key R\&D Program of China} (No. 2023YFA1010700). The research of Gianmarco Sperone is currently supported by the \textit{Chilean National Agency for Research and Development} (ANID) through the \textit{Fondecyt Iniciaci\'on} grant 11250322.  \par\smallskip
\noindent
{\bf Data availability statement.} Data sharing not applicable to this article as no datasets were generated or analyzed during the current study.
\par\smallskip
\noindent
{\bf Conflict of interest statement}.  The authors declare that they have no conflict of interest.

\bibliographystyle{emsplain}

\newpage
\noindent
\hspace{0.1mm}
\begin{minipage}{140mm}
\textbf{Giovanni Paolo Galdi}\\
Department of Mechanical Engineering and Materials Science\\
University of Pittsburgh\\
Benedum Engineering Hall 607\\
Pittsburgh, PA 15261 - USA\\
E-mail: galdi@pitt.edu
\vspace{0.3cm}	
\end{minipage}
\newline
\vspace{0.5cm}
\noindent
\begin{minipage}{140mm}
\textbf{Filippo Gazzola}\\
Dipartimento di Matematica\\
Politecnico di Milano\\
Piazza Leonardo da Vinci 32\\
20133 Milan - Italy\\
E-mail: filippo.gazzola@polimi.it	
\end{minipage}
\newline
\vspace{0.5cm}
\noindent
\begin{minipage}{100mm}
\textbf{Mikhail Korobkov}\\
School of Mathematical Sciences\\
Fudan University\\
Handan Road 220\\
200433 Shanghai - People's Republic of China\\
and\\
Sobolev Institute of Mathematics\\
Siberian Branch of the Russian Academy of Sciences\\
Akademika Koptyuga Prospekt 4\\
630090 Novosibirsk - Russia\\
E-mail: korob@math.nsc.ru
\end{minipage}
\newline
\vspace{0.5cm}
\noindent
\begin{minipage}{100mm}
\textbf{Xiao Ren}\\
Fudan University\\
Handan Road 220\\
200433 Shanghai - People's Republic of China\\
E-mail: xren@fudan.edu.cn
\end{minipage}
\newline
\vspace{0.5cm}
\noindent
\begin{minipage}{100mm}
\textbf{Gianmarco Sperone}\\
Facultad de Matemáticas\\
Pontificia Universidad Católica de Chile\\
Avenida Vicuña Mackenna 4860\\
7820436 Santiago - Chile\\
E-mail: gianmarco.sperone@uc.cl
\end{minipage}

\end{document}